\documentclass[11pt,reqno]{amsart} 

\usepackage[displaymath, mathlines]{lineno}

\usepackage{graphicx}

\usepackage{microtype}

\usepackage[top=1.5in,bottom=1.5in,left=1.5in,right=1.5in]{geometry} 

\usepackage{xcolor}

\usepackage{enumitem}

\usepackage[labelfont=bf,font=sf]{caption}

\usepackage[font={sf,small},labelfont=md]{subfig}

\usepackage[noadjust,nosort]{cite}

\usepackage{amsmath, amsthm, amssymb, tikz, mathtools, array, mathrsfs, tensor}

\usepackage{polynom}

\usepackage{esint}

\usepackage{multicol}

\usepackage{dsfont}
	\newcommand{\one}{\mathds{1}}

\usepackage{framed}

\usepackage{upref}

\usepackage{hyperref}
	\hypersetup{colorlinks=true,linkcolor=blue,citecolor=blue,urlcolor=black} 

\allowdisplaybreaks

\numberwithin{equation}{section}


\newcommand{\eq}[1]{\begin{linenomath}\postdisplaypenalty=0\begin{align*} #1 \end{align*}\end{linenomath}}
\newcommand{\eeq}[1]{\begin{linenomath}\postdisplaypenalty=0\begin{align} \begin{split} #1 \end{split} \end{align}\end{linenomath}}
\newcommand{\eeqs}[2]{\begin{subequations}\label{#1}\begin{linenomath}\postdisplaypenalty=0\begin{align} #2 \end{align}\end{linenomath}\end{subequations}}

\newcommand{\stackref}[2]{\stackrel{\mbox{\footnotesize{\eqref{#1}}}}{#2}}
\newcommand{\stackrefp}[2]{\stackrel{\phantom{\mbox{\footnotesize{\eqref{#1}}}}}{#2}}

\def\eps{\varepsilon}
\def\vphi{\varphi}

\newcommand{\E}{\mathbb{E}}

\renewcommand{\P}{\mathbb{P}}

\newcommand{\R}{\mathbb{R}}

\newcommand{\Z}{\mathbb{Z}}

\newcommand{\CC}{\mathcal{C}}

\newcommand{\FF}{\mathcal{F}}

\newcommand{\HH}{\mathcal{H}}

\newcommand{\PP}{\mathcal{P}}

\newcommand{\RR}{\mathcal{R}}

\newcommand{\ZZ}{\mathcal{Z}}

\newcommand{\DDD}{\mathscr{D}}

\newcommand{\PPP}{\mathscr{P}}

\newcommand{\Esf}{\mathsf{E}}

\newcommand{\vc}[1]{{\boldsymbol #1}}

\newcommand{\wt}[1]{\widetilde{#1}}
\newcommand{\wh}[1]{\widehat{#1}}

\DeclareMathOperator{\Var}{Var}
\DeclareMathOperator{\Cov}{Cov}
\DeclareMathOperator{\Corr}{Corr}

\newcommand{\givenp}[3][]{#1( #2 \: #1| \: #3 #1)} 

\DeclareMathOperator{\e}{e} 
\newcommand{\cc}{\mathrm{c}} 
\newcommand{\dd}{\mathrm{d}} 

            \makeatletter
            \DeclareFontFamily{OMX}{MnSymbolE}{}
            \DeclareSymbolFont{MnLargeSymbols}{OMX}{MnSymbolE}{m}{N}
            \SetSymbolFont{MnLargeSymbols}{bold}{OMX}{MnSymbolE}{b}{N}
            \DeclareFontShape{OMX}{MnSymbolE}{m}{N}{
                <-6>  MnSymbolE5
               <6-7>  MnSymbolE6
               <7-8>  MnSymbolE7
               <8-9>  MnSymbolE8
               <9-10> MnSymbolE9
              <10-12> MnSymbolE10
              <12->   MnSymbolE12
            }{}
            \DeclareFontShape{OMX}{MnSymbolE}{b}{N}{
                <-6>  MnSymbolE-Bold5
               <6-7>  MnSymbolE-Bold6
               <7-8>  MnSymbolE-Bold7
               <8-9>  MnSymbolE-Bold8
               <9-10> MnSymbolE-Bold9
              <10-12> MnSymbolE-Bold10
              <12->   MnSymbolE-Bold12
            }{}
            
            \let\llangle\@undefined
            \let\rrangle\@undefined
            \DeclareMathDelimiter{\llangle}{\mathopen}%
                                 {MnLargeSymbols}{'164}{MnLargeSymbols}{'164}
            \DeclareMathDelimiter{\rrangle}{\mathclose}%
                                 {MnLargeSymbols}{'171}{MnLargeSymbols}{'171}
            \makeatother
            
    \DeclareFontFamily{U}{matha}{\hyphenchar\font45}
    \DeclareFontShape{U}{matha}{m}{N}{ <-6> matha5 <6-7> matha6 <7-8>
    matha7 <8-9> matha8 <9-10> matha9 <10-12> matha10 <12-> matha12 }{}
    \DeclareSymbolFont{matha}{U}{matha}{m}{N}
    \DeclareFontFamily{U}{mathx}{\hyphenchar\font45}
    \DeclareFontShape{U}{mathx}{m}{N}{ <-6> mathx5 <6-7> mathx6 <7-8>
    mathx7 <8-9> mathx8 <9-10> mathx9 <10-12> mathx10 <12-> mathx12 }{}
    \DeclareSymbolFont{mathx}{U}{mathx}{m}{N}
    
    \DeclareMathDelimiter{\llbrack} {4}{matha}{"76}{mathx}{"30}
    \DeclareMathDelimiter{\rrbrack} {5}{matha}{"77}{mathx}{"38}
    
\newcommand{\Pb}{\mathbf{P}}

\newtheorem{thm}{Theorem}[section]
\newtheorem{prop}[thm]{Proposition}
\newtheorem{cor}[thm]{Corollary}
\newtheorem{lemma}[thm]{Lemma}
\newtheorem{claim}[thm]{Claim}
\newtheorem{theirthm}{Theorem} 

\theoremstyle{definition}

\newtheorem{remark}[thm]{Remark}

\renewcommand{\thefootnote}{\fnsymbol{footnote}}

\title{Full-path localization of directed polymers}



\author{Erik Bates}
\thanks{This research was supported by NSF grant DMS-1902734} 
\address{\newline Department of Mathematics \newline University of Wisconsin--Madison \newline Van Vleck Hall \newline 480 Lincoln Drive \newline Madison, WI 53706-1324 \newline United States \newline \textup{\tt ewbates@wisc.edu}}

\begin{document}
\bibliographystyle{acm}

\renewcommand{\thefootnote}{\arabic{footnote}} \setcounter{footnote}{0}

\begin{abstract}
Certain polymer models are known to exhibit path localization in the sense that at low temperatures, the average fractional overlap of two independent samples from the Gibbs measure is bounded away from $0$.
Nevertheless, the question of where along the path this overlap takes place has remained unaddressed.
In this article, we prove that on linear scales, overlap occurs along the entire length of the polymer.
Namely, we consider time intervals of length $\varepsilon N$, where $\varepsilon>0$ is fixed but arbitrarily small.
We then identify a constant number of distinguished trajectories such that the Gibbs measure is concentrated on paths having, with one of these distinguished paths, a fixed positive overlap simultaneously in every such interval.
This result is obtained in all dimensions for a Gaussian random environment by using a recent non-local result as a key input.
\end{abstract}

\maketitle


\section{Introduction}
In statistical physics, the phenomenon of \textit{localization} refers to the tendency of disordered systems, especially at low temperatures, to revert to one of a small number of energetically favorable states, even as the size of the system diverges.
Beginning with Anderson's formative work \cite{anderson58}, it has been a general goal to describe conditions (\textit{e.g.}~the presence of random impurities, random interaction strengths, or random geometry) under which localization occurs.
Often, for a given model, a main challenge is to characterize free energy non-analyticities---which may be already difficult to rigorously detect---as separators between non-localized and localized phases.
If this can be done, it gives rise to the task of more precisely quantifying the system's behavior in either regime, the localized phase being more physically anomalous and thus harder to predict.

This paper focuses on such questions for \textit{directed polymers in random environment}.
Defined in the next section, this model was introduced by Huse and Henley \cite{huse-henley85} to study interfaces of the Ising model subject to random impurities, and later adopted in the mathematics literature by Imbrie and Spencer \cite{imbrie-spencer88} as a model for polymer growth in random media.
At low temperatures, directed polymers exhibit localization properties which have been a frequent object of study over the last forty years; a nearly complete survey is provided in the book of Comets \cite{comets17}, and related models are discussed in \cite{denhollander09}.

Most of the literature on localization has focused on the polymer's endpoint distribution, but very recently there has been progress in proving pathwise localization.
For certain random environments at sufficiently low temperature, it is now known that if two polymers are sampled independently under the same environment, then with non-vanishing probability they will intersect for a non-vanishing fraction of their length.
However, owing to the global nature of this property, the results to date provide little information on the local structure needed to produce this effect; for instance, \textit{where} these intersections occur.
The main purpose of this paper is to provide a first result in this direction, stated as Theorem \ref{main_thm} in Section \ref{main_result}.

Interestingly, the central input to the proof is a recent \textit{non}-local path localization result from \cite{bates-chatterjee20II}.
This plan of attack is natural from the perspective of random walks (from which polymers are defined), whose structure of i.i.d.~increments frequently allows one to translate between local and global information.
For directed polymers, however, there is no obvious renewal feature to function in the same way.
Fortunately, we identify as a weak surrogate a multi-temperature free energy expression \eqref{generalized_free_energy_thm_2} that permits one to analyze isolated segments of the polymer.
This technique is summarized in Section \ref{proof_sketch} and may be of independent interest.

\subsection{The model: directed polymers in Gaussian environment} 
Let $\sigma = (\sigma_i)_{i\geq0}$ denote simple random walk on $\Z^d$ starting at the origin.
We will write $P$ to denote the law of $\sigma$ in the space
\eeq{ \label{sigma_def}
\Sigma \coloneqq \{\sigma = (\sigma_i)_{i\geq0} \in (\Z^d)^{\{0,1,\dots\}}:\, \sigma_0 = 0, \|\sigma_i - \sigma_{i-1}\| = 1 \text{ for all $i\geq1$}\},
}
equipped with the standard cylindrical sigma-algebra.
Expectation with respect to $P$ will be denoted by $E(\cdot)$.

Next let $\vc g = (g(i,x) :\, i\geq1,x\in\Z^d)$ be a collection of i.i.d.~standard normal random variables, supported on some probability space $(\Omega,\FF,\P)$.
Expectation according to $\P$ will be denoted by $\E(\cdot)$.
The infinite collection $\vc g$ is called the \textit{disorder} or \textit{random environment}, and defines a family of Hamiltonians on $\Sigma$,
\eq{
H_N(\sigma) \coloneqq \sum_{i=1}^N g(i,\sigma_i), \quad N\geq1.
}
At inverse temperature $\beta\geq0$, the associated Gibbsian \textit{polymer measure} is given by
\eeq{ \label{mu_def}
\mu_{N,\beta}(\dd\sigma) \coloneqq \frac{1}{Z_N(\beta)}\e^{\beta H_N(\sigma)}\ P(\dd\sigma),
}
where
$Z_N(\beta) \coloneqq E(\e^{\beta H_N(\sigma)})$
is the random normalization constant known as the \textit{partition function}.
As a function of $N$, the partition function grows exponentially with a limiting rate $p(\beta)$ called the \textit{free energy}.

\begin{theirthm} \label{free_energy_thm}
\textup{\cite[Rmk.~2.1]{comets17}} 
There exists a function $p \colon [0,\infty)\to\R$ such that
\eeq{ \label{free_energy_thm_1}
\lim_{N\to\infty} \frac{\E \log Z_N(\beta)}{N} = p(\beta) \quad \text{for all $\beta\geq0$.}
}
Moreover, for any $u>0$ we have
\eeq{ \label{free_energy_thm_2}
\P\Big(\Big|\frac{\log Z_N(\beta)}{N} - \frac{\E\log Z_N(\beta)}{N}\Big| \geq u\Big) \leq 2\exp\Big(\frac{-Nu^2}{2\beta^2}\Big).
}
Consequently, the following limit holds for every $\beta\geq0$:
\eeq{ \label{free_energy_thm_3}
\lim_{N\to\infty} \frac{\log Z_N(\beta)}{N} = p(\beta) \quad \text{$\P$-$\mathrm{a.s.}$ and in $L^\alpha(\P)$ for all $\alpha\in[1,\infty)$}.
}
\end{theirthm}

Given this paper's methods, the following observation will help avoid some technical concerns.

%

\begin{remark} \label{independence_remark}
\textit{A priori}, the validity of \eqref{free_energy_thm_3} might depend on the fact that the random variables defining $H_N(\cdot)$ also appear in $H_M(\cdot)$ for $M\geq N$.
On the contrary, because of \eqref{free_energy_thm_2}, the statement \eqref{free_energy_thm_3} is still true if one takes
\eeq{ \label{new_H}
H_N(\sigma) = \sum_{i=1}^N g_N(i,\sigma_i),
}
where now $\vc g = (g_N(i,x) :\, 1\leq i\leq N, x\in\Z^d)$ is an i.i.d.~collection even across $N$.
Henceforth, we will take \eqref{new_H} as the definition of $H_N$.
The distribution of $\mu_{N,\beta}$ does not change; only the joint law of $(\mu_{N,\beta})_{N\geq1}$ is affected, and we will not be concerned with the latter object.
\end{remark}

We will be interested in the relationship between $p(\beta)$ and the \textit{overlap function},
\eq{
R(\sigma^1,\sigma^2) \coloneqq \frac{1}{N}\sum_{i=1}^N \one_{\{\sigma_i^1=\sigma_i^2\}}, \quad \sigma^1,\sigma^2\in\Sigma,
}
where the dependence of $R(\cdot,\cdot)$ on $N$ is understood.
The degree to which the model localizes can be measured by the typical size of $R(\sigma^1,\sigma^2)$ when $\sigma^1$ and $\sigma^2$ are sampled independently from $\mu_{N,\beta}$.
For instance, if $\beta = 0$, then $\mu_{N,0}$ returns the simple random walk $P$, and classical results give the overlap's rate of decay:
\eq{
R(\sigma^1,\sigma^2)  \asymp \begin{cases} N^{-1/2} &\text{if }d=1, \\
N^{-1}\log N &\text{if }d=2,\\
N^{-1} &\text{if }d\geq 3.
\end{cases}
}
Considering that $R(\sigma^1,\sigma^2)\to0$ as $N\to\infty$ in any one of these cases, it is a striking fact that when disorder is introduced at sufficiently large $\beta>0$, this overlap remains bounded away from 0 (in various senses made precise in Section \ref{background}).
As suggested earlier, the free energy provides an understanding of this dichotomy as a phase transition between high and low temperatures.
In the following statements, the function $\beta^2/2$ appears because it is the logarithmic moment generating function of the standard normal distribution.

\begin{theirthm} \label{critical_temperature}
\textup{\cite[Thm.~3.2]{comets-yoshida06}}
There exists a critical inverse temperature $\beta_\cc = \beta_\cc(d) \in [0,\infty)$ such that
\eq{
0 \le \beta \leq \beta_{\mathrm{c}} \quad &\implies \quad p(\beta) = \beta^2/2, \\
\beta > \beta_{\mathrm{c}} \quad &\implies \quad p(\beta) < \beta^2/2.
}
\end{theirthm}

The high temperature phase $0\leq\beta<\beta_\cc$ is thought to indicate a polymer measure still resembling simple random walk; a result to this effect is \cite[Thm.~1.2]{comets-yoshida06}.
On the other hand, in the low temperature phase $\beta>\beta_\cc$, the polymer measure is expected to be so attracted by favorable regions in the random environment that it concentrates near them.
The question then is how to relate the condition $p(\beta) < \beta^2/2$ to this localization, as measured by the overlap function.


\begin{remark} \label{rmk_conj}
The function $\log Z_N(\cdot)$ is a logarithmic moment generating function and thus convex.
It thus follows from \eqref{free_energy_thm_1} that $p(\cdot)$ is also convex and hence differentiable almost everywhere.
It is believed (see \cite[Conj.~6.1]{comets17}) that there are actually no points of non-differentiability, and moreover that $p'(\beta) < \beta$ for all $\beta > \beta_\cc$.
If this is true, then $p'(\beta) < \beta$ is equivalent to the low-temperature condition $p(\beta) < \beta^2/2$.
\end{remark}

\subsection{Background} \label{background}
The model we have defined makes sense if $\vc g$ is replaced by any family of disorder variables.
The i.i.d.~assumption is completely standard, and it is only out of methodological necessity that we have assumed Gaussianity.
The Gaussian case happens to be one of the few for which some version of path localization has been rigorously established, but the phenomenon is anticipated in much greater generality.

The first path localization result for \eqref{mu_def} appeared in \cite[Thm.~6.1]{comets17}, although the relevant computation was already present in the work of Carmona and Hu \cite[Lem.~7.1]{carmona-hu02}.
Adopting a Gaussian-integration-by-parts idea used in continuous models \cite{comets-cranston13,comets-yoshida13} and earlier in the spin glass literature \cite{aizenman-lebowitz-ruelle87,comets-neveu95,talagrand06III,panchenko08}, one can show\footnote{The identity \eqref{expected_R_limit} is verified in full, for a very general Gaussian disordered system, in \cite[Cor.~3.10]{bates-chatterjee20II}.} that if $p(\cdot)$ is differentiable at $\beta$, then
\eeq{ \label{expected_R_limit}
\lim_{N\to\infty} \E\bigg[\int_{\Sigma\times\Sigma} R(\sigma^1,\sigma^2)\ \mu_{N,\beta}^{\otimes 2}(\dd\sigma^1,\dd\sigma^2)\bigg] = 1 - \frac{p'(\beta)}{\beta}.
}
In particular, when $p'(\beta) < \beta$ (by Remark \ref{rmk_conj}, this is the presumed characterization of low temperature), the average overlap between independent polymer paths has a nonzero limiting expectation.
In other words, if $\sigma^1$ and $\sigma^2$ are sampled independently from $\mu_{N,\beta}$, then there is a nonzero chance that their fractional overlap $R(\sigma^1,\sigma^2)$ is at least some fixed positive number.
For continuous models, analogous results can be found in \cite{comets-cranston13,comets-yoshida13} as well as \cite[Sec.~5.5]{comets-cosco??}.

For as elegantly simple as \eqref{expected_R_limit} is to prove, it only tells us that the previous sentence is true on an event of nonzero $\P$-probability.
One should like said probability to be asymptotically equal to $1$, meaning the specific realization of the disorder is irrelevant. Such was the advancement provided by Chatterjee \cite{chatterjee19}, for sufficiently large $\beta$ and a certain class of bounded random environments.
In \cite{bates-chatterjee20II}, Bates and Chatterjee proved an analogous (but less quantitative) statement in the Gaussian case, and then bootstrapped that result to the following, stronger one.
The key feature is that the number $J$ of distinguished paths has no dependence on the polymer length $N$, although the distinguished paths themselves, called $\sigma^1,\dots,\sigma^J$, are random and do depend on $N$.

\begin{theirthm}  \label{previous_thm_polymer}
\textup{\cite[Thm.~1.6]{bates-chatterjee20II}}
Assume $\beta>0$ is a point of differentiability for $p(\cdot)$ with $p'(\beta) < \beta$.
Then for every $\eps > 0$, there exist integers $J = J(\beta,\eps)$, $N_* = N_*(\beta,\eps)$ and a number $\delta = \delta(\beta,\eps)>0$ such that the following is true for all $N\geq N_*$.
With $\P$-probability at least $1-\eps$, there are paths $\sigma^1,\dots,\sigma^J\in\Sigma$ satisfying
\eq{
\mu_{N,\beta}\bigg(\bigcup_{j=1}^J \big\{\sigma\in\Sigma:\, R(\sigma^{j},\sigma)\geq\delta\big\}\bigg) \geq 1 - \eps.
}
\end{theirthm}

In this sense, $\mu_{N,\beta}$ concentrates on highly frequented paths and places no mass elsewhere.
The distinguished trajectories $\sigma^1,\dots,\sigma^J$ (which are random and depend on $N$) might be called ``favorite paths'', representing the preferred regions or ``favorite corridors'' of the polymer measure $\mu_{N,\beta}$.

While this brief overview has mentioned essentially all that has been proved about path localization (at least for the discrete model considered in this paper), much more is known about localization of the endpoint distribution $\mu_{N,\beta}(\sigma_N \in \cdot)$.
The state of the art goes well beyond the Gaussian case or even simple random walks (on the latter point, see \cite{bates18,bakhtin-seo20,viveros?} and references therein), and there is even a one-dimensional exactly solvable model \cite{seppalainen12} admitting an explicit limiting law for the endpoint distribution \cite{comets-nguyen16}.
The reader is referred to \cite{bates-chatterjee20} for a review of the literature.

Finally, a somewhat orthogonal direction of work considers directed polymers in heavy-tailed random environments, mostly in $d=1$.
In this setting, the degree of localization is much greater (\textit{e.g.}~\cite[Thm.~2.1]{auffinger-louidor11}), and so the interesting questions arise from taking $\beta = \beta_N \to 0$, where the rate of decay is determined by the index of the heavy tail \cite{dey-zygouras16,torri16,berger-torri19,berger-lacoin21}.
Further discussion can be found in \cite[Sec.~6.4]{comets17}; see also \cite{vargas07}.

\subsection{Main result} \label{main_result}
The goal of this article is to go beyond the single statistic $R(\sigma^1,\sigma^2)$.
Although it serves as a natural gauge for localization, it does little to illuminate the geometry of localized polymers.
For instance, if we know $R(\sigma^1,\sigma^2)$ is bounded away from zero, can we say something about the set of $i \in \{1,2,\dots,N\}$ for which $\sigma_i^1 = \sigma_i^2$?
Our main result addresses this question.

For integers $a\leq b$, let $\llbrack a,b\rrbrack$ denote the integer interval $\{a,a+1,\dots,b\}$.
Given $1\leq a\leq b\leq N$, consider the \textit{restricted overlap},
\eq{
R^{\llbrack a,b\rrbrack}(\sigma^1,\sigma^2) \coloneqq \frac{1}{b-a+1}\sum_{i=a}^b \one_{\{\sigma^1_i = \sigma^2_i\}}, \quad \sigma^1,\sigma^2\in\Sigma.
}
By examining these restricted overlaps, we will prove that the intersection set mentioned above is dense in $\llbrack 1,N\rrbrack$.
Mirroring the language of Theorem \ref{previous_thm_polymer}, we make this assertion precise as follows. 

\begin{thm} \label{main_thm}
Assume $\beta>0$ is a point of differentiability for $p(\cdot)$ with $p'(\beta) < \beta$.
Then for every $\eps > 0$, there exist integers $J = J(\beta,\eps)$, $N_* = N_*(\beta,\eps)$ and a number $\delta = \delta(\beta,\eps)>0$ such that the following is true for all $N\geq N_*$.
With $\P$-probability at least $1-\eps$, there are paths $\sigma^1,\dots,\sigma^J\in\Sigma$ satisfying
\eq{
\mu_{N,\beta}\bigg(\bigcup_{j=1}^J \big\{\sigma\in\Sigma:\, R^{\llbrack a,b\rrbrack}(\sigma^{j},\sigma)\geq\delta\  \mathrm{whenever}\ b-a+1 \geq \eps N\big\}\bigg) \geq 1 - \eps.
}
\end{thm}

So at low temperatures and up to negligible events, a sample from the polymer measure localizes around one of a fixed number of distinguished paths, and this localization takes place along the entire length of the path; that is, in every interval of size at least $\eps N$.
It is this latter part that is the contribution of the present article.

An important comment is that the statement of Theorem \ref{main_thm} concerns fixed $\beta$, which can be arbitrarily close to $\beta_\cc$.
Prior to this result, it was only possible to make guarantees about localization away from the polymer's endpoint if $\beta$ were sent to $\infty$
(\textit{c.f.}~\cite[Eq.~(4)]{comets-cranston13}, \cite[Thm.~3.3.3 \& 3.3.4]{comets-yoshida13}, and \cite[Sec.~9]{comets-cosco??}).
Indeed, because $p'(\cdot)$ is bounded (see \cite[Prop.~2.1(iii)]{comets17}), the identity \eqref{expected_R_limit} implies localization at a large fraction of times when $\beta$ is sufficiently large, but even this leaves open the possibility of some linearly sized interval on which localization does not occur.
Theorem \ref{main_thm} rules out this behavior at all low temperatures satisfying $p'(\beta) < \beta$.

A difficult and important question left open is the optimal dependence of $J$ and $\delta$ on $\eps$.
The proof method in this paper likely leads to very poor bounds.
Furthermore, can one prove localization on scales finer than linear?
This would be a necessary step toward showing that, at least in $d=1$, the polymer measure concentrates on a favorite corridor of $O(1)$ width; see \cite[Sec.~12.9(6)]{denhollander09}.

\subsection{Outline of proof} \label{proof_sketch}
Consider a weaker version of Theorem \ref{main_thm} in which we replace general subintervals $\llbrack a,b\rrbrack$ by
only ``regular'' subintervals: $(0,\frac{N}{L}]$, $(\frac{N}{L},\frac{2N}{L}]$, $\dots$, $(\frac{(L-1)N}{L},N]$, where $L$ is some large integer.
Our first observation is that Theorem \ref{main_thm} will be implied by this special case, which is stated as Theorem \ref{weak_main_thm}.
Indeed, for a given $\eps>0$, we can choose $L$ large enough that the regular subintervals are somewhat smaller than $\eps N$ and thus actually contained in any interval $I$ of size $\eps N$.
In this way, positive overlap in the regular subintervals will imply positive overlap in $I$.
This is the content of Section \ref{first_reduction}.

Another difficulty of Theorem \ref{main_thm} is that we demand the same distinguished path $\sigma^j$ to be used in every subinterval $\llbrack a,b\rrbrack$ of appropriate size.
The steps of the previous paragraph do not remove this requirement, and so our second reduction is to a version of Theorem \ref{weak_main_thm} that allows the index $j$ to depend on which regular subinterval $(\frac{(\ell-1)N}{N},\frac{\ell N}{L}]$ is considered.
This yet weaker result is stated as Theorem \ref{weaker_main_thm}, and the reduction argument is given in Section \ref{weaker_to_weak}.
The rough idea is to concatenate segments of distinguished paths in order to produce a larger set but still of $O(1)$ size, so that whenever a path $\sigma$ had intersected two distinct distinguished paths in consecutive subintervals, it will now intersect a single concatenated path in both subintervals.
This procedure can be carried out by demanding slightly less overlap in each regular subinterval.

Having made these reductions, we are left to prove that for each regular subinterval, one can (with high probability) find a bounded number of paths such that a sample from the Gibbs measure will (with high probability) have non-vanishing overlap in the given subinterval with at least one of these paths.
This statement could be easily proved if one were able to apply Theorem \ref{previous_thm_polymer} within each subinterval and then take an appropriate union bound.
The seeming obstruction is that the marginal of $\mu_{N,\beta}$ in a given subinterval is not a polymer measure of the same form as $\mu_{N,\beta}$.
Moreover, this marginal depends on the environment at all times, not just those within the subinterval.
Nevertheless, we can regard these marginals as polymer measures with respect to \textit{random} reference measures.
That is, we replace $P$ in \eqref{mu_def} by a random measure which, crucially, is determined entirely by the environment \textit{outside} the given subinterval.
Correspondingly, the Hamiltonian $H_N(\cdot)$ is replaced by a sum depending only on the environment inside said subinterval, the remaining disorder having been absorbed into the random reference measure.

In this setup, Theorem \ref{previous_thm_polymer} still does not quite apply because it assumes a specific reference measure $P$.
Fortunately, we can appeal to a more general result of \cite{bates-chatterjee20II} from which Theorem \ref{previous_thm_polymer} was derived.
We recall this general result as Theorem \ref{previous_thm} in Section \ref{weaker_proof}.
The only hypothesis to check is that $\mu_{N,\beta}$ still admits a limiting free energy with respect to the random reference measures.
To prove this fact, we introduce in Section \ref{multi_temp} a ``multi-temperature free energy'' that, as a special case, can ignore the disorder in a given subinterval.
Convergence of this generalized free energy is stated in Theorem \ref{generalized_free_energy_thm} and proved using modifications of standard techniques.
Finally, Theorem \ref{previous_thm} is invoked in Section \ref{weaker_proof}, where further technical issues are addressed en route to proving Theorem \ref{weaker_main_thm}.

\section{Reduction to regular subintervals} \label{first_reduction}

Given positive integers $N$ and $L$, let $0 = n_0(N)\leq n_1(N) \leq \cdots \leq n_L(N) = N$ be any sequence satisfying
\eeq{ \label{interval_condition}
n_\ell(N) - n_{\ell-1}(N) \in \Big\{\Big\lfloor \frac{N}{L} \Big\rfloor,\Big\lceil \frac{N}{L} \Big\rceil\Big\} \quad \text{for all $\ell=1,\dots,L$}.
}
We will think of $L$ as fixed throughout, and then such a sequence will be chosen and fixed for each $N$.
In other words, we partition the integer interval $\llbrack 1,N\rrbrack$ into $L$ parts of the form $\llbrack n_{\ell-1}(N)+1, n_\ell(N)\rrbrack$, whose sizes are as close to equal as possible.
For the sake of exposition, let us call these parts \textit{regular subintervals}.
The fractional overlap between $\sigma^1,\sigma^2\in\Sigma$ in the $\ell^\text{th}$ subinterval will be denoted by
\eeq{ \label{ell_overlap_def}
R^{(\ell)}(\sigma^1,\sigma^2) &\coloneqq R^{\llbrack n_{\ell-1}(N)+1,n_\ell(N)\rrbrack}(\sigma^1,\sigma^2) \\
&= \frac{1}{n_\ell(N)-n_{\ell-1}(N)} \sum_{i=n_{\ell-1}(N)+1}^{n_\ell(N)} \one_{\{\sigma^1_i=\sigma^2_i\}}.
}
When we are not varying $N$, we will simply write $n_\ell$ in place of $n_\ell(N)$. 

The following special case of Theorem \ref{main_thm} will allow us to prove the general case.

\begin{thm} \label{weak_main_thm}
Assume $\beta>0$ is a point of differentiability for $p(\cdot)$ with $p'(\beta) < \beta$.
Then for every $\eps > 0$ and positive integer $L$, there exist integers $J = J(\beta,\eps,L)$, $N_* = N_*(\beta,\eps,L)$ and a number $\delta = \delta(\beta,\eps,L)>0$ such that the following is true for all $N\geq N_*$.
With $\P$-probability at least $1-\eps$, there are paths $\sigma^1,\dots,\sigma^J\in\Sigma$ satisfying
\eq{
\mu_{N,\beta}\bigg(\bigcup_{j=1}^J \bigcap_{\ell=1}^L\big\{\sigma\in\Sigma:\, R^{(\ell)}(\sigma^{j},\sigma)\geq\delta\big\}\bigg) \geq 1 - \eps.
}
\end{thm}

Given this result, we now show that Theorem \ref{main_thm} readily follows by identifying regular subintervals lying within a given interval $\llbrack a,b\rrbrack$ of size at least $\eps N$.

\begin{proof}[Proof of Theorem \ref{main_thm}]
Given $\eps\in(0,1]$, let $L\geq2$ be the integer satisfying $2/L \leq \eps < 2/(L-1)$.
Consider any subinterval $\llbrack a,b\rrbrack \subset \llbrack 1,N\rrbrack$ of size $b-a+1 \geq \eps N$.
Our choice of $L$ guarantees that $\llbrack n_{\ell-1}(N)+1,n_{\ell}(N)\rrbrack \subset \llbrack a,b\rrbrack$ for some $\ell \in \llbrack 1,L\rrbrack$.

Let us first address the case when $b-a+1\leq 2\eps N+1$.
In particular, assuming $N\geq L$, we have 
\eq{
b-a+1< \frac{4N}{L-1}+1 \leq \frac{8N}{L}+1 \leq \frac{9N}{L}.
}
Asymptotically we know $(n_{\ell}-n_{\ell-1}) \sim N/L$ as $N\to\infty$, but let us just use the trivial bounds
\eeq{ \label{trivial_bd_1}
\frac{N}{2L} 
\leq\Big\lfloor \frac{N}{L}\Big\rfloor
\leq n_{\ell}-n_{\ell-1}
\leq \Big\lceil \frac{N}{L}\Big\rceil
\leq \frac{2N}{L}
\quad \text{for all $N\geq L$}.
}
For any $\sigma^1,\sigma^2 \in\Sigma$, the inclusion $\llbrack n_{\ell-1}+1,n_{\ell}\rrbrack \subset \llbrack a,b\rrbrack$ now gives
\eq{
\frac{1}{n_{\ell}-n_{\ell-1}} \sum_{i=n_{\ell-1}+1}^{n_\ell} \one_{\{\sigma^1_i=\sigma^2_i\}}
\leq \frac{2L}{N} \sum_{i=a}^{b} \one_{\{\sigma^1_i=\sigma^2_i\}}
\leq \frac{18}{b-a+1}\sum_{i=a}^{b} \one_{\{\sigma^1_i=\sigma^2_i\}}.
}
Therefore, the following implication is true:
\eq{
R^{(\ell)}(\sigma^j,\sigma) \geq \delta \quad \implies \quad R^{\llbrack a,b\rrbrack}(\sigma^j,\sigma) \geq \delta/18.
}
If $b-a+1> 2\eps N+1$, then we can partition $\llbrack a,b\rrbrack$ into disjoint subintervals, all having sizes at least $\eps N$ but no larger than $2\eps N+1$.
The argument from above applies to each of these subintervals, and so
\eq{
R^{(\ell)}(\sigma^j,\sigma) \geq \delta \quad \text{for all $\ell=1,\dots,L$} \quad \implies \quad
R^{\llbrack a,b\rrbrack}(\sigma^j,\sigma) \geq \delta/18.
}
We can now deduce Theorem \ref{main_thm} from Theorem \ref{weak_main_thm} by replacing $\delta$ with $\delta/18$.
\end{proof}

\section{Reduction to independent subintervals} \label{weaker_to_weak}
We continue using the notation of regular subintervals introduced in Section \ref{first_reduction}, where the task of proving Theorem \ref{main_thm} was reduced to showing Theorem \ref{weak_main_thm}.
In this section, we reduce Theorem \ref{weak_main_thm} to the following, yet weaker statement. 

\begin{thm} \label{weaker_main_thm}
Assume $\beta>0$ is a point of differentiability for $p(\cdot)$ with $p'(\beta) < \beta$.
Then for every $\eps > 0$ and positive integer $L$, there exist integers $J = J(\beta,\eps,L)$, $N_* = N_*(\beta,\eps,L)$ and a number $\delta = \delta(\beta,\eps,L)>0$ such that the following is true for all $N\geq N_*$. 
With $\P$-probability at least $1-\eps$, there are paths $\sigma^1,\dots,\sigma^J\in\Sigma$ satisfying
\eq{
\mu_{N,\beta}\bigg(\bigcap_{\ell=1}^L\bigcup_{j=1}^J\big\{\sigma\in\Sigma:\, R^{(\ell)}(\sigma^{j},\sigma)\geq\delta\big\}\bigg) \geq 1 - \eps.
}
\end{thm}

Assuming this result, it remains to show that the same overlapping distinguished path can be taken in all $L$ regular subintervals (\textit{i.e.}~exchanging the intersection and the union displayed above).
We now argue that this can be done by increasing $J$ an $O(1)$ amount and choosing $\delta$ appropriately smaller.


\begin{proof}[Proof of Theorem \ref{weak_main_thm}]
Let $\eps > 0$ and a positive integer $L$ be given.
Then take $J$, $N_*$, and $\delta\in(0,1]$ as in Theorem \ref{weaker_main_thm} so that for all $N\geq N_*$, the following event occurs with $\P$-probability at least $1-\eps$.
There exists a random set of paths $\DDD_1 = \{\sigma^1,\dots,\sigma^J\}\subset\Sigma$ such that
\eeq{ \label{original_mu_statement}
\mu_{N,\beta}\bigg(\bigcap_{\ell=1}^L\bigcup_{\sigma'\in\DDD_1} \big\{\sigma\in\Sigma:\, R^{(\ell)}(\sigma',\sigma)\geq\delta\big\}\bigg) \geq 1 - \eps.
}
We will henceforth assume this event occurs.

Set $K \coloneqq \lceil 12/\delta \rceil$.
By possibly making $N_*$ larger, we may assume $N$ is such that
\eq{ 
\lfloor N/L\rfloor \geq K.
}
We note for later that this assumption implies
\eeq{ \label{N_assumption_consequence}
\Big\lceil \frac{\lceil N/L\rceil}{K}\Big\rceil \stackref{trivial_bd_1}{\leq} \Big\lceil \frac{ 2N/L}{K}\Big\rceil \leq \frac{ 3N/L}{K},
}
and also that $\delta \leq 1$ implies
\eeq{ \label{delta_assumption_consequence}
\frac{12}{\delta} \leq K \leq \frac{12}{\delta}+1\leq \frac{13}{\delta}.
}
Given $\DDD_1$, let us perform the following inductive procedure.

For each $\ell = 1,\dots,L$, partition the interval $\llbrack n_{\ell-1}+1,n_\ell\rrbrack$ into $K$ subintervals $\llbrack m^{(\ell)}_{k-1}+1,m^{(\ell)}_k\rrbrack$, $k=1,\dots,K$, whose sizes are as close to equal as possible.
That is, we choose a sequence
\eeqs{subsubinterval}{
n_{\ell-1} = m^{(\ell)}_0 \leq m^{(\ell)}_1 \leq \cdots \leq m^{(\ell)}_{K} = n_\ell \qquad \qquad \quad
\intertext{satisfying}
m^{(\ell)}_{k} - m^{(\ell)}_{k-1} \in \Big\{\Big\lfloor \frac{n_{\ell}-n_{\ell-1}}{K}\Big\rfloor,\Big\lceil \frac{n_{\ell}-n_{\ell-1}}{K}\Big\rceil\Big\} \quad \text{for all $k = 1,\dots,K$}.
}
For $\ell \leq L-1$, set $\DDD_{\ell+1} = \DDD_{\ell} \cup \DDD_\ell^+$, where the supplementary set $\DDD_\ell^+$ is formed as follows.
Consider every ordered pair $(\sigma',\sigma'') \in \DDD_{\ell}\times\DDD_{\ell}$ with $\sigma'\neq \sigma''$.
For each $k \in \llbrack 1,K-1\rrbrack$, and each $t \in \llbrack n_{\ell}+1,n_{\ell+1}\rrbrack$, determine whether there exists a nearest-neighbor path between $\sigma'_{m_k^{(\ell)}}$ and $\sigma''_t$ consisting of exactly $t - m_k^{(\ell)}$ steps.
Let $t^{(\ell)}_k(\sigma',\sigma'')$ be the minimal such $t$; more formally,
\eeq{ \label{meeting_time_def}
t^{(\ell)}_k(\sigma',\sigma'') \coloneqq \inf\{t \in \llbrack n_{\ell}+1,n_{\ell+1}\rrbrack :\, \exists\, \sigma\in\Sigma, \sigma_{m_k^{(\ell)}} = \sigma'_{m_k^{(\ell)}}, \sigma_t = \sigma''_t\}.
}
If $t_k^{(\ell)}(\sigma',\sigma'')=\infty$, then we do nothing further.
Otherwise, there is some nearest-neighbor path $\wt\sigma$ connecting $\sigma'_{m_k^{(\ell)}}$ and $\sigma''_{t^{(\ell)}_k(\sigma',\sigma'')}$ in exactly $t^{(\ell)}_k(\sigma',\sigma'') - m_k^{(\ell)}$ steps. 
(If there are multiple such paths, then chose one according to some deterministic rule.)
We include the following concatenated path, which we call $\CC^{(\ell)}_k(\sigma',\sigma'')$, as an element of $\DDD_\ell^+$ (see Figure \ref{figure_1}):
\eeq{ \label{new_path}
0 = \sigma'_0\quad\stackrel{\text{first $m^{(\ell)}_k$ steps of $\sigma'$}}{\mapsto}\quad \sigma'_{m_k^{(\ell)}} \quad\stackrel{\text{follow $\wt\sigma$}}{\mapsto} \quad\sigma_{t^{(\ell)}_k(\sigma',\sigma'')}''\quad \stackrel{\text{continue on $\sigma''$}}{\mapsto}\cdots
}

\begin{figure}
\fbox{\includegraphics[width=0.88\textwidth]{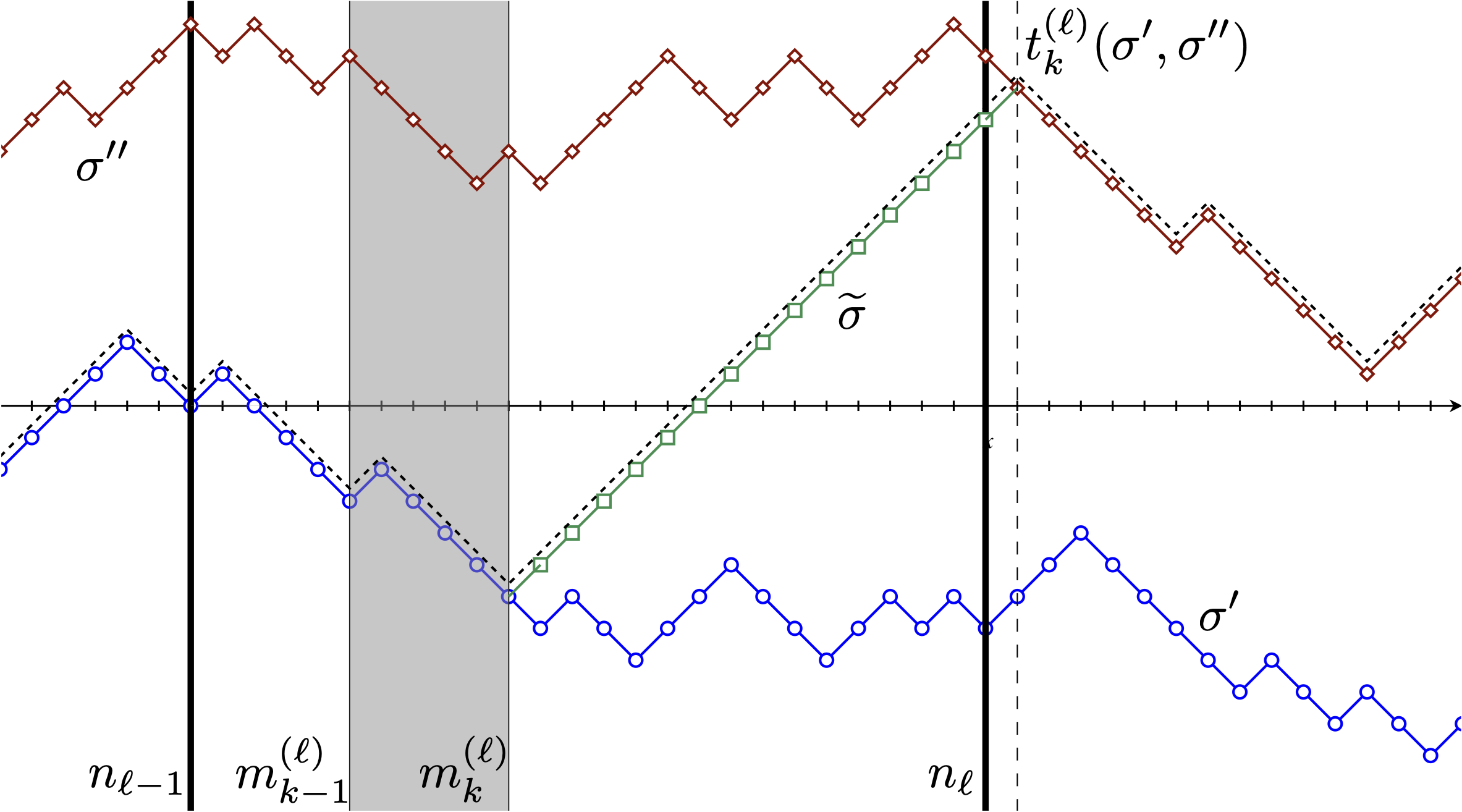}}
\captionsetup{width=0.9\textwidth}
\caption{Example construction of $\CC_k^{(\ell)}(\sigma',\sigma'')$ in $d=1$.
Time is visualized in the horizontal direction.
The path $\sigma'\in\DDD_\ell$ is shown as the blue solid curve with circles, $\sigma''\in\DDD_\ell$ as the red solid curve with diamonds, and the connecting path $\wt\sigma$ as the green solid curve with squares.
The concatenated path $\CC_k^{(\ell)}(\sigma',\sigma'')\in\DDD_{\ell+1}$ is displayed as the dashed trajectory, with a dashed vertical line marking the time $t_k^{(\ell)}(\sigma',\sigma'')$ of earliest possible connection from $\sigma'_{m_k^{(\ell)}}$ to $\sigma''$ terminating in the time interval $\llbrack n_{\ell}+1,n_{\ell+1}\rrbrack$.}
\label{figure_1}
\end{figure}

Once this procedure has been performed for all ordered pairs $(\sigma',\sigma'') \in \DDD_{\ell}\times\DDD_{\ell}$ with $\sigma'\neq\sigma''$, the construction of the set $\DDD_\ell^+$ is complete.
Note that $|\DDD_\ell^+| \leq |\DDD_{\ell}|(|\DDD_{\ell}|-1)(K-1)$, and so $|\DDD_{\ell+1}| \leq K|\DDD_{\ell}|^2$, which leads to the upper bound
\eq{
|\DDD_L| \leq K^{2^{L-1}-1}|\DDD_1|^{2^{L-1}} \stackref{delta_assumption_consequence}{\leq}
\Big(\frac{13}{\delta}\Big)^{2^{L-1}-1} J^{2^{L-1}}.
}
In particular, $|\DDD_L|$ is bounded by a constant independent of $N$.

We now claim that
\eeq{ \label{desired_containment}
\bigcup_{\sigma' \in \DDD_L}\bigcap_{\ell=1}^L \{\sigma\in\Sigma :\, R^{(\ell)}(\sigma',\sigma) \geq \delta^2/104\} \supset \bigcap_{\ell=1}^L\bigcup_{\sigma'\in\DDD_1} \{\sigma\in\Sigma :\, R^{(\ell)}(\sigma',\sigma) \geq \delta\}.
}
In light of \eqref{original_mu_statement} and the earlier observation regarding the cardinality of $\DDD_L$, the containment \eqref{desired_containment} establishes the conclusion of Theorem \ref{weak_main_thm} after replacing $J$ by $\big(\frac{13}{\delta}\big)^{2^{L-1}-1} J^{2^{L-1}}$ and $\delta$ by $\delta^2/104$.
In order to prove \eqref{desired_containment}, we reduce to the following claim.

\begin{claim} \label{inductive_reduction}
Suppose $\sigma \in \Sigma$ is such that for some $\ell\in\llbrack 1,L-1\rrbrack$, there exist $\sigma',\sigma'' \in \DDD_\ell$ such that
\eeqs{induct_hypothesis}{
R^{(\ell)}(\sigma',\sigma) &\geq \delta, \label{induct_1} \\
R^{(\ell+1)}(\sigma'',\sigma) &\geq \delta. \label{induct_2}
}
Then there is $\sigma'''\in\DDD_{\ell+1}$ such that 
\eeqs{induct_conclusion}{
R^{(\ell')}(\sigma''',\sigma) &= R^{(\ell')}(\sigma',\sigma) \quad \text{for all $\ell' \in \llbrack 1,\ell-1\rrbrack$}, \label{induct_3} \\
R^{(\ell)}(\sigma''',\sigma) &\geq \delta^2/104, \label{induct_4} \\
R^{(\ell+1)}(\sigma''',\sigma) &\geq \delta. \label{induct_5}
}
\end{claim}

Indeed, assume that Claim \ref{inductive_reduction} holds.
Any $\sigma$ belonging to the right-hand side of \eqref{desired_containment} has $R^{(1)}(\sigma^1,\sigma) \geq \delta$ and $R^{(2)}(\sigma^2,\sigma) \geq \delta$ for some $\sigma^1,\sigma^2\in\DDD_1$.
So there is $\sigma^{1,2} \in \DDD_2$ such that  $R^{(1)}(\sigma^{1,2},\sigma) \geq \delta^2/104$ and $R^{(2)}(\sigma^{1,2},\sigma) \geq \delta$.
Since there is also $\sigma^3\in\DDD_1\subset\DDD_2$ satisfying $R^{(3)}(\sigma^3,\sigma) \geq \delta$, we can repeat the process to produce $\sigma^{1,2,3}\in\DDD_3$ satisfying $R^{(1)}(\sigma^{1,2,3},\sigma), R^{(2)}(\sigma^{1,2,3},\sigma)\geq \delta^2/104$ and $R^{(3)}(\sigma^{1,2,3},\sigma) \geq \delta$.
Continuing in this way, one arrives at $\sigma^{1,\dots,L}\in\DDD_L$ such that $R^{(\ell)}(\sigma^{1,\dots,L},\sigma) \geq \delta^2/104$ for all $\ell\in\llbrack 1,L\rrbrack$.
That is, $\sigma$ belongs to the left-hand side of \eqref{desired_containment}.
\end{proof}

\begin{proof}[Proof of Claim \ref{inductive_reduction}]

Let $\sigma\in\Sigma$ and assume \eqref{induct_hypothesis} for some $\sigma',\sigma''\in\DDD_\ell$.
If $\sigma'=\sigma''$, then \eqref{induct_conclusion} holds by taking $\sigma'''$ equal to $\sigma'=\sigma''$.

If instead $\sigma'\neq\sigma''$, we recall the subintervals of $\llbrack m_{k-1}^{(\ell)}+1,m_k^{(\ell)}\rrbrack$, $k\in\llbrack 1,K\rrbrack$, introduced in \eqref{subsubinterval}.
We have
\eq{
\sum_{k=1}^{K}\sum_{i = m^{(\ell)}_{k-1}+1}^{m_k^{(\ell)}} \one_{\{\sigma'_i=\sigma_i\}}
= \sum_{i=n_{\ell-1}+1}^{n_\ell}  \one_{\{\sigma'_i=\sigma_i\}}
= (n_\ell-n_{\ell-1})R^{(\ell)}(\sigma',\sigma)
\stackref{trivial_bd_1}{\geq} \frac{\delta N}{2L}.
}
Therefore, there must be at least two distinct values of $k\in\llbrack 1,K\rrbrack$ for which
\eeq{ \label{smaller_guarantee}
\sum_{i = m^{(\ell)}_{k-1}+1}^{m^{(\ell)}_k} \one_{\{\sigma'_i = \sigma_i\}} \geq \frac{\delta N}{4LK},
}
since otherwise we would have the contradictory bound
\eq{
\sum_{k=1}^{K}\sum_{i= m^{(\ell)}_{k-1}+1}^{m^{(\ell)}_k} \one_{\{\sigma'_i = \sigma_i\}}
< (K-1)\frac{\delta N}{4LK} + \Big\lceil \frac{n_{\ell}-n_{\ell-1}}{K} \Big\rceil
\stackref{N_assumption_consequence}{\leq} \frac{\delta N}{4L} +\frac{ 3N/L}{K}
\stackref{delta_assumption_consequence}{\leq} \frac{\delta N}{2L}.
}
Let us call these two values $k_1$ and $k_2$, where $k_1 < k_2$.
We will now argue that the time $t_{k_1}^{(\ell)}(\sigma',\sigma'')$ defined in \eqref{meeting_time_def} is finite, and that the path $\sigma''' = \CC_{k_1}^{(\ell)}(\sigma',\sigma'')$ defined in \eqref{new_path} satisfies \eqref{induct_conclusion}.

Now, we know that $\sigma''_t = \sigma_t$ for some $t\in\llbrack n_{\ell}+1,n_{\ell+1}\rrbrack$, simply by the fact that $R^{(\ell+1)}(\sigma'',\sigma) > 0$.
We claim that any such $t$ must satisfy $t \geq t^{(\ell)}_{k_1}(\sigma',\sigma'')$. (In particular, $t^{(\ell)}_{k_1}(\sigma',\sigma'')$ is finite.)
Indeed, by \eqref{smaller_guarantee} there is some $s \in \llbrack m_{k_2-1}^{(\ell)}+1,m^{(\ell)}_{k_2}\rrbrack$ for which $\sigma'_s = \sigma_s$.
Therefore, by following $\sigma'$ from $\sigma'_{m_{k_1}^{(\ell)}}$ to $\sigma'_s = \sigma_s$, and then following $\sigma$ from $\sigma_s$ to $\sigma_{t}=\sigma''_{t}$, we will have constructed a nearest-neighbor path connecting $\sigma'_{m^{(\ell)}_{k_1}}$ to $\sigma''_{t}$ in exactly $t - m^{(\ell)}_{k_1}$ steps; see Figure \ref{figure_2}.
\begin{figure}
\fbox{\includegraphics[width=0.88\textwidth]{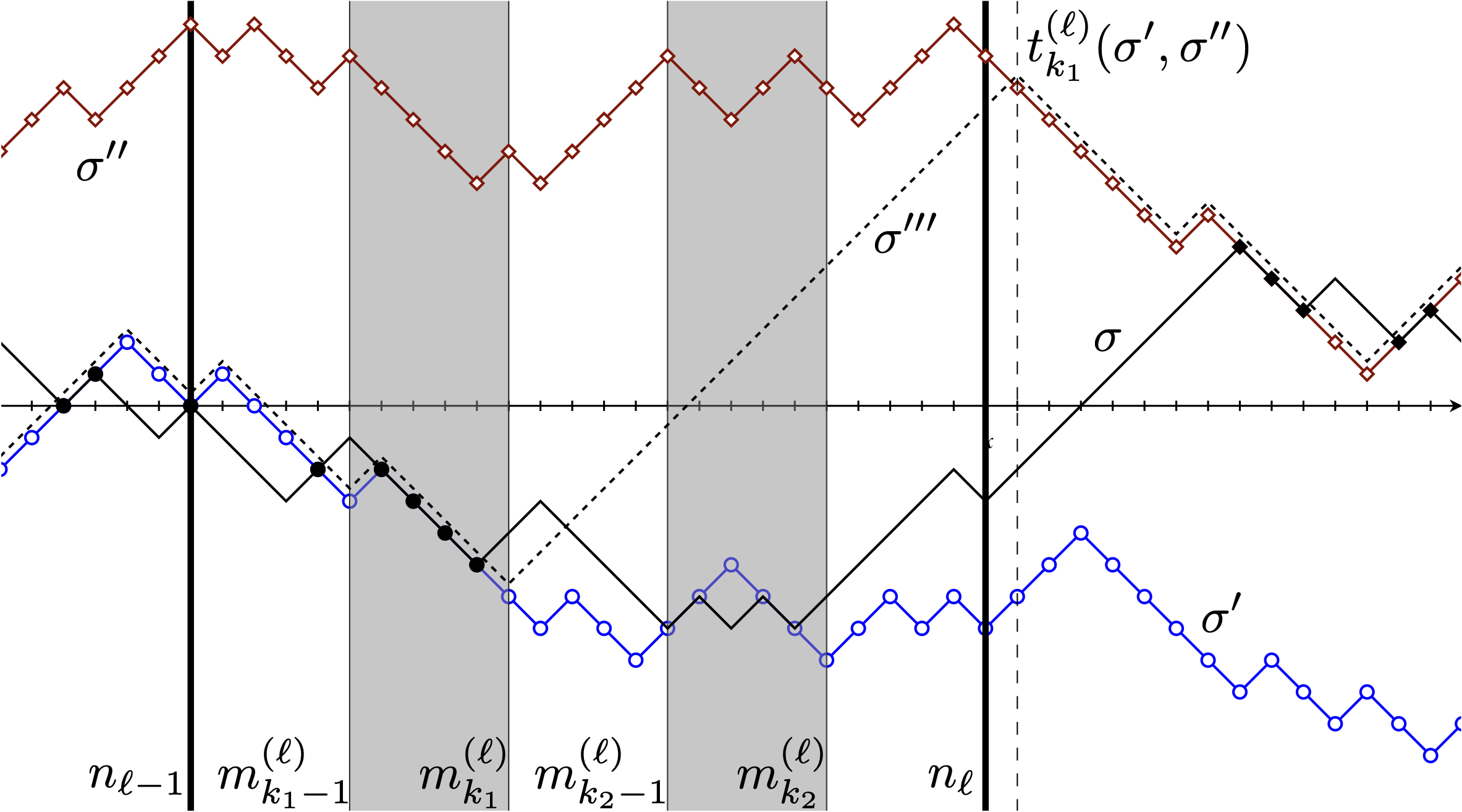}}
\captionsetup{width=0.9\textwidth}
\caption{Overlap considerations in Claim \ref{inductive_reduction}.
As in Figure \ref{figure_1}, $\sigma'$ and $\sigma''$ are displayed as solid curves with circles and diamonds, respectively, and $\sigma''' = \CC_{k_1}^{(\ell)}(\sigma',\sigma'')$ is the dashed trajectory. 
The path $\sigma$ under consideration is shown as a solid curve without decoration, and its intersections with $\sigma'''$ are marked as filled circles or diamonds.
On one hand, because $\sigma'''$ agrees with $\sigma'$ at least until time $m_{k_1}^{(\ell)}$, it retains all overlap with $\sigma$ incurred by $\sigma'$ up to that point; this observation leads to \eqref{induct_3} and \eqref{induct_4}.
On the other, $\sigma$ is known to intersect with $\sigma'$ in some later interval $\llbrack m_{k_2-1}^{(\ell)}+1,m_{k_2}^{(\ell)}\rrbrack$, and so all its intersections with $\sigma''$ past time $n_{\ell}$ must occur after $\sigma'''$ has coincided with $\sigma''$, which occurs at time $t_{k_1}^{(\ell)}(\sigma',\sigma'')$; hence \eqref{induct_5} follows from \eqref{induct_2}.}
\label{figure_2}
\end{figure}
By definition \eqref{meeting_time_def}, this means $t \geq t^{(\ell)}_{k_1}(\sigma',\sigma'')$.

Given that $t_{k_1}^{(\ell)}(\sigma',\sigma'')$ is finite, let $\sigma'''$ be the path $\CC_{k_1}^{(\ell)}(\sigma',\sigma'')$  constructed in \eqref{new_path}.
In particular, $\sigma'''$ agrees with $\sigma'$ up to time $m_{k_1}^{(\ell)}$, and with $\sigma''$ from time $t_{k_1}^{(\ell)}(\sigma',\sigma'')$ onward.
In symbols, these are the statments
\begin{linenomath}\postdisplaypenalty=0
\begin{align}
i \in\llbrack1,m_{k_1}^{(\ell)}\rrbrack \quad &\implies \quad \sigma'''_i = \sigma'_i, \label{agree_1} \\
i \in\llbrack t_{k_1}^{(\ell)}(\sigma',\sigma''),N\rrbrack \quad &\implies \quad \sigma'''_i = \sigma''_i. \label{agree_2}
\end{align}
\end{linenomath}
Now, the argument of the previous paragraph showed that for any $t\in\llbrack n_\ell+1,n_{\ell+1}\rrbrack$ such that $\sigma_t'' = \sigma_t$, we necessarily have $t \geq t_{k_1}^{(\ell)}(\sigma',\sigma'')$ and thus $\sigma'''_t=\sigma''_t=\sigma_t$ by \eqref{agree_2}.
In particular, we have $R^{(\ell+1)}(\sigma''',\sigma) \geq R^{(\ell+1)}(\sigma'',\sigma) \geq \delta$, thus verifying \eqref{induct_5}.
On the other hand, \eqref{induct_3} follows from \eqref{agree_1} because $m^{(\ell)}_{k_1} > n_{\ell-1}$.
Finally, to obtain \eqref{induct_4}, observe that
\eq{ 
R^{(\ell)}(\sigma''',\sigma)
&\stackref{agree_1}{\geq} \frac{1}{n_{\ell}-n_{\ell-1}}\sum_{i = m^{(\ell)}_{k_1-1}+1}^{m^{(\ell)}_{k_1}}  \one_{\{\sigma'_i = \sigma_i\}} \\
&\stackref{smaller_guarantee}{\geq} \frac{\delta N}{(n_\ell-n_{\ell-1})4LK}
\stackref{trivial_bd_1}{\geq} \frac{\delta}{8K} 
\stackref{delta_assumption_consequence}{\geq} \frac{\delta^2}{104}.
\tag*{\qedhere}
}
\end{proof}

\section{Multi-temperature free energy} \label{multi_temp}

As outlined in Section \ref{proof_sketch}, our proof strategy in Section \ref{weaker_proof} 
will require us to isolate the Hamiltonian on each regular subinterval $\llbrack n_{\ell-1}+1,n_\ell\rrbrack$.
Mechanically, this can be done by letting the inverse temperature $\beta$ depend on time and setting it equal to zero outside the interval $\llbrack n_{\ell-1}+1,n_\ell\rrbrack$.
As it turns out, it will be easier to take the complementary route of setting the inverse temperature to zero \textit{inside} $\llbrack n_{\ell-1}+1,n_\ell\rrbrack$, and keeping it unchanged outside.
Either choice changes the free energy, of course, and we will need to show that a statement analogous to $\eqref{free_energy_thm_3}$ still holds.
It will not be any more difficult, however, to allow the inverse temperature to assume a different value on each interval $\llbrack n_{\ell-1}+1,n_\ell\rrbrack$, $\ell\in\llbrack 1,L\rrbrack$.

In what follows, we will use the notation $\PP_{N,L}$ to denote the partition 
\eeq{ \label{partition_def}
\PP_{N,L}: \quad 0 = n_0(N) \leq n_1(N) \leq \cdots \leq n_L(N) = N,
}
which is chosen to satisfy \eqref{interval_condition}.
Let $\vc\beta = (\beta_1,\dots,\beta_L) \in [0,\infty)^L$ and consider the Hamiltonian 
\eq{
H_{\PP_{N,L}}^{\vc\beta}(\sigma) \coloneqq \sum_{\ell=1}^L \beta_\ell\sum_{i=n_{\ell-1}+1}^{n_\ell}g_N(i,\sigma_i).
}
The associated 
partition function will be written as
\eq{
Z_{\PP_{N,L}}(\vc\beta) \coloneqq E(\e^{H_{\PP_{N,L}}^{\vc\beta}(\sigma)}).
}
The concentration result stated below shows that the multi-temperature expression \linebreak $\frac{1}{N}\log Z_{P_{N,L}}(\vc\beta)$ is well-approximated by an average of single-temperature free energies.
Ultimately, we will need only the asymptotic statement \eqref{generalized_free_energy_thm_2}, but with minimal extra effort we can prove \eqref{generalized_free_energy_thm_1} as an intermediate step.

\begin{thm} \label{generalized_free_energy_thm}
For any $L\geq1$, there exist positive constants $C_1 = C_1(L,d)$ and $c_2 = c_2(L)$ such that
for any $u>0$, $N\geq L^2$, $\vc\beta\in[0,\infty)^L$, and $\PP_{N,L}$ satisfying \eqref{interval_condition}, we have
\eeq{ \label{generalized_free_energy_thm_1}
\P\Big(\Big|\frac{\log Z_{\PP_{N,L}}(\vc\beta)}{N} - \sum_{\ell=1}^L \frac{\E\log Z_{n_\ell-n_{\ell-1}}(\beta_\ell)}{N} \Big| > u\Big) \leq C_1N^d\exp\Big(-c_2\frac{Nu^2}{\beta_{\max}^2}\Big),
}
where $\beta_{\max} \coloneqq \max\{\beta_1,\dots,\beta_L\}$.
Consequently,
\eeq{ \label{generalized_free_energy_thm_2}
\lim_{N\to\infty} \frac{\log Z_{\PP_{N,L}}(\vc\beta)}{N} = \frac{1}{L}\sum_{\ell=1}^L p(\beta_\ell) \quad \text{$\P$-$\mathrm{a.s.}$ and in $L^\alpha(\P)$ for all $\alpha\in[1,\infty)$}.
}
\end{thm}

In the proof, we will use the following simple lemma.

\begin{lemma} \label{trivial_lemma}
Let $X$ be a random variable taking values in $(a,b)$ with mean $\Upsilon$, where $-\infty \leq a < b \leq \infty$.
Let $f \colon (a,b)\to\R$ be a monotone function, and denote the mean of $f(X)$ by $\Upsilon_f$.
If $|f(\Upsilon)-\Upsilon_f| > 2u > 0$, then the event $\{|f(X) - y| > u\}$ has positive probability for every $y\in\R$.
\end{lemma}

\begin{proof}
We prove the contrapositive.
That is, assume $f(X)\in[y-u,y+u]$ with probability one, for some $y\in\R$.
Without loss of generality, we may assume $f$ is non-decreasing; if not, we apply the argument to $-f$.
If $X$ is an almost sure constant, then $f(\Upsilon) = \Upsilon_f$.
Otherwise, there exists $\delta>0$ such that each of $\{X\leq\Upsilon-\delta\}$ and $\{X\geq\Upsilon+\delta\}$ occur with positive probability.
In this case, we must have 
\eq{
y-u \leq f(\Upsilon-\delta) \leq f(\Upsilon)\leq f(\Upsilon+\delta)\leq y+u.
}
Of course, we also know $\Upsilon_f \in [y-u,y+u]$, and so $|f(\Upsilon)-\Upsilon_f| \leq 2u$.
\end{proof}


\begin{proof}[Proof of Theorem \ref{generalized_free_energy_thm}]
We proceed by induction on $L$. 
Our inductive hypothesis is that for any $u>0$, any integer $N\geq L^2$, and any partition 
\eq{
\PP_{N,L}: \quad 0 = n_0(N) < n_1(N) < \cdots <  n_{L}(N) = N
}
that is valid in the sense of \eqref{interval_condition}, we have
\eeq{
\label{first_summands}
&\P\bigg(\bigg|\frac{\log Z_{\PP_{N,L}}(\vc\beta)}{N} - \sum_{\ell=1}^{L}\frac{\E\log Z_{ n_\ell- n_{\ell-1}}(\beta_\ell)}{N}\bigg| > u\bigg)\\ 
&\leq 4L(2N+1)^d\exp\Big(-\frac{Nu^2}{18L^2\beta^2_{\max}}\Big),
}
where \eqref{free_energy_thm_2} provides the base case of $L=1$.
Once we prove the inductive step, \eqref{first_summands} will yield \eqref{generalized_free_energy_thm_1} with  $C_1 = 3^d\cdot 4L$ and $c_2 = 1/(18L^2)$.
Then, by standard arguments, it follows that
\eq{
\lim_{N\to\infty} \Big|\frac{\log Z_{\PP_{N,L}}(\vc\beta)}{N} - \sum_{\ell=1}^L \frac{\E\log Z_{n_\ell-n_{\ell-1}}(\beta_\ell)}{N} \Big| = 0 \quad \text{$\P$-$\mathrm{a.s.}$ and in $L^\alpha(\P)$, $\alpha\in[1,\infty)$}.
}
This limit is seen to be equivalent to \eqref{generalized_free_energy_thm_2} once we recall \eqref{free_energy_thm_1}.
Therefore, the rest of the proof is establishing the induction.
We thus assume \eqref{first_summands} and consider any $\beta_{L+1}\in[0,\infty)$ and any partition
\eq{
\PP_{N,L+1}: \quad 0 = n_0(N) < n_1(N) < \cdots < n_L(N) < n_{L+1}(N) = N,
}
where now $N\geq(L+1)^2$.

Define the set $D_i \coloneqq \{x\in\Z^d :\, P(\sigma_i=x) > 0\}$ for each integer $i\geq1$.
Observe that for any fixed realization of the disorder $\vc g$, if we condition on the value of $\sigma_i$ for some $i\in\llbrack 1,N\rrbrack$, then by the Markov property of the random walk, the vectors
\eq{
(g_N(1,\sigma_1),\dots,g_N(i,\sigma_i)) \quad \text{and} \quad (g_N(i,\sigma_i),\dots,g_N(N,\sigma_N))
}
are conditionally independent with respect to $P$.
Using this observation when $i = n_{L}$, we have
\eq{
&Z_{\PP_{N,L+1}}(\vc \beta,\beta_{L+1}) = \sum_{x\in\Z^d} E\givenp{\e^{H_{\PP_{N,L+1}}^{(\vc \beta,\beta_{L+1})}(\sigma)}}{\sigma_{n_{L}}=x}P(\sigma_{n_{L}}=x) \\
&=\sum_{x\in D_{n_{L}}} \bigg[E\givenp[\bigg]{\exp\bigg\{\sum_{\ell=1}^{L}\beta_\ell\sum_{i=n_{\ell-1}+1}^{n_\ell}g_N(i,\sigma_i)\bigg\}}{{\sigma_{n_{L}}=x}}P(\sigma_{n_{L}}=x) \\
&\phantom{=} \times E\givenp[\bigg]{\exp\bigg\{\beta_{L+1}\sum_{i=n_{L}+1}^{n_{L+1}} g_N(i,\sigma_i)\bigg\}}{\sigma_{n_{L}}=x}\bigg].
}
To condense notation, let us write
\eq{
A(x) &\coloneqq E\bigg(\exp\bigg\{\sum_{\ell=1}^{L}\beta_\ell\sum_{i=n_{\ell-1}+1}^{n_\ell}g_N(i,\sigma_i)\bigg\};{\sigma_{n_{L}}=x}\bigg), \\
B(x) &\coloneqq E\givenp[\bigg]{\exp\bigg\{\beta_{L+1}\sum_{i=n_{L}+1}^{n_{L+1}} g_N(i,\sigma_i)\bigg\}}{\sigma_{n_{L}}=x}.
}
Note that
\eq{
A \coloneqq \sum_{x\in D_{n_{L}}} A(x)
&=Z_{\bar\PP_{N,L}}(\beta_1,\dots,\beta_{L}),
}
where $\bar\PP_{n_L,L}$ is the partition of $\llbrack 1,n_{L}\rrbrack$ into $L$ parts induced by $\PP_{N,L+1}$.
That is,
\eq{
\bar\PP_{n_L,L}: \ \ 0 = \bar n_0(n_L) \leq \bar n_1(n_{L}) \leq \cdots \leq \bar n_{L}(n_{L}) = n_{L}, \ \ \text{where} \ \
\bar n_\ell(n_{L}(N)) = n_\ell(N).
}
We are thus interested in the limit of 
\eeq{ \label{peeling_off}
\frac{\log Z_{\PP_{N,L+1}}(\vc\beta,\beta_{L+1})}{N} 
&= \frac{1}{N}\log\bigg(A\sum_{x\in D_{n_{L}}}\frac{A(x)}{A}B(x)\bigg) \\
&= \frac{\log Z_{\bar\PP_{n_L,L}}(\vc\beta)}{N} + \frac{1}{N}\log\sum_{x\in D_{n_{L}}} \frac{A(x)}{A}B(x).
}
Since $N\geq (L+1)^2$, we necessarily have $N/(L+1) \geq L+1$ and thus
\eeq{ \label{annoying_1}
\frac{\lceil N/(L+1)\rceil}{\lfloor N/(L+1)\rfloor} \leq \frac{L+2}{L+1}.
}
Therefore,
\eeq{ \label{last_term_bound}
n_{L} \geq N - \Big\lceil \frac{N}{L+1} \Big\rceil 
&\geq N-\frac{L+2}{L+1}\Big\lfloor\frac{N}{L+1} \Big\rfloor \\
&\geq  N-\frac{L+2}{L+1}\frac{N}{L+1} \\
&= \frac{N}{(L+1)^2}((L+1)^2-L-2) \geq \frac{N}{(L+1)^2}L^2 \geq L^2,
}
and so the first term in the final expression of \eqref{peeling_off} is subject to the concentration inequality from \eqref{first_summands}.
That is,
\eeq{ \label{induction_raw}
&\P\bigg(\bigg|\frac{\log Z_{\bar\PP_{n_L,L}}(\vc\beta)}{N} - \sum_{\ell=1}^{L}\frac{\E\log Z_{\bar n_\ell-\bar n_{\ell-1}}(\beta_\ell)}{N}\bigg| > u\bigg) \\
&\stackrefp{last_term_bound}{=} \P\bigg(\bigg|\frac{\log Z_{\bar\PP_{n_L,L}}(\vc\beta)}{n_{L}} - \sum_{\ell=1}^{L}\frac{\E\log Z_{ n_\ell- n_{\ell-1}}(\beta_\ell)}{n_{L}}\bigg| > \frac{N}{n_{L}}u\bigg) \\
&\stackrefp{last_term_bound}{\leq} 4L(2n_{L}+1)^d\exp\Big(-\frac{n_{L}\big(\frac{N}{n_{L}}u\big)^2}{18L^2\beta^2_{\max}}\Big) \\
&\stackref{last_term_bound}{\leq} 4L(2N+1)^d\exp\Big(-\frac{N\big(\frac{N}{n_{L}}u\big)^2}{18(L+1)^2\beta^2_{\max}}\Big).
}
We are now left with the task of controlling the second term in the final expression of \eqref{peeling_off}.

Since all variables in $\vc g$ are i.i.d., we have the following equality in law with respect to $\P$:
\eeq{ \label{free_energy_in_distribution}
B(x)
 \stackrel{\mathrm{dist}}{=}Z_{N-n_{L}}(\beta_{L+1}) \quad \text{for each $x\in D_{n_{L}}$}.
}
In particular, $\E\log B(x)$ is constant among such $x$, and so taking this constant as the value of $y$, we conclude the following from Lemma \ref{trivial_lemma} with $f(t) = \log t$ and $X$ having probability distribution given by $A(\cdot)/A$.
If
\eq{
&\bigg|\log \sum_{x\in D_{n_{L}}} \frac{A(x)}{A}B(x) - \sum_{x\in D_{n_{L}}}\frac{A(x)}{A}\log B(x)\bigg|
> 2u,
}
then
\eq{
|\log B(x) - \E \log B(x)| > u \quad \text{for some $x\in D_{n_{L}}$.}
}
Also note that the following holds for all $\ell\in\llbrack 1,L+1\rrbrack$, in particular $\ell=L+1$:
\eq{
n_{\ell}-n_{\ell-1} \geq \Big\lfloor \frac{N}{L+1}\Big\rfloor \stackref{annoying_1}{\geq} \frac{L+1}{L+2}\Big\lceil\frac{N}{L+1}\Big\rceil 
\geq \frac{N}{L+2} 
\geq \frac{N}{(L+1)^2}.
}
Consequently,
\eeq{ \label{concentration_1}
&\P\bigg(\frac{1}{N}\bigg|\log \sum_{x\in D_{n_{L}}} \frac{A(x)}{A}B(x) - \sum_{x\in D_{n_{L}}}\frac{A(x)}{A}\log B(x)\bigg|
> 2u\bigg) \\
&\stackrel{\phantom{\mbox{\footnotesize\eqref{free_energy_in_distribution},\eqref{free_energy_thm_2}}}}{\leq} \P\bigg(\bigcup_{x\in D_{n_{L}}} \Big\{\frac{|\log B(x) - \E\log B(x)|}{N} > u\Big\}\bigg) \\
&\stackrel{\mbox{\footnotesize\eqref{free_energy_in_distribution},\eqref{free_energy_thm_2}}}{\leq} 2|D_{n_{L}}|\exp\Big(-\frac{(N-n_{L})\big(u\frac{N}{N-n_{L}}\big)^2}{2\beta_{L+1}^2}\Big) \\
&\stackrel{\phantom{\mbox{\footnotesize\eqref{free_energy_in_distribution},\eqref{free_energy_thm_2}}}}{\leq}  2(2N+1)^d\exp\Big(-\frac{N\big(u\frac{N}{N-n_{L}}\big)^2}{2(L+1)^2\beta_{L+1}^2}\Big).
}
Moreover, by writing
\eq{
&\bigg|\E\log Z_{N-n_{L}}(\beta_{L+1}) - \sum_{x\in D_{n_{L}}}\frac{A(x)}{A}\log B(x)\bigg| \\
&\leq \sum_{x\in D_{n_{L}}}\frac{A(x)}{A}|\E\log B(x) - \log B(x)|,
}
we can repeat the previous estimate to obtain
\eeq{ \label{concentration_2}
&\P\bigg(\frac{1}{N}\bigg|\E\log Z_{N-n_{L}}(\beta_{L+1}) - \sum_{x\in D_{n_{L}}}\frac{A(x)}{A}\log B(x)\bigg| > u\bigg) \\
&\leq \P\bigg(\bigcup_{x\in D_{n_{L}}} \Big\{\frac{|\log B(x) - \E\log B(x)|}{N} > u\Big\}\bigg) \\
&\leq 2(2N+1)^d\exp\Big(-\frac{N\big(u\frac{N}{N-n_{L}}\big)^2}{2(L+1)^2\beta_{L+1}^2}\Big).
}
Together, \eqref{concentration_1} and \eqref{concentration_2} yield
\eeq{ \label{concentration_3}
&\P\bigg(\frac{1}{N}\bigg|\E\log Z_{N-n_{L}}(\beta_{L+1}) - \log \sum_{x\in D_{n_{L}}} \frac{A(x)}{A}B(x)\bigg| > 3u\bigg) \\
&\leq 4(2N+1)^d\exp\Big(-\frac{N\big(u\frac{N}{N-n_{L}}\big)^2}{2(L+1)^2\beta_{L+1}^2}\Big).
}
Putting together \eqref{peeling_off}, \eqref{induction_raw}, and \eqref{concentration_3}, we conclude
\eq{
&\P\Big(\Big|\frac{\log Z_{\PP_{N,L+1}}(\vc\beta,\beta_{L+1})}{N} - \sum_{\ell=1}^{L+1} \frac{\E\log Z_{n_\ell-n_{\ell-1}}(\beta_\ell)}{N} \Big| > u\Big) \\
&\leq \P\bigg(\bigg|\frac{\log Z_{\bar\PP_{n_L,L}}(\vc\beta)}{N} - \sum_{\ell=1}^{L}\frac{\E\log Z_{\bar n_\ell-\bar n_{\ell-1}}(\beta_\ell)}{N}\bigg| > u\frac{n_{L}}{N}\bigg) \\
&\phantom{\leq} + \P\bigg(\frac{1}{N}\bigg|\E\log Z_{N-n_{L}}(\beta_{L+1}) - \log \sum_{x\in D_{n_{L}}} \frac{A(x)}{A}B(x)\bigg| > u\frac{N-n_{L}}{N}\bigg) \\
&\leq 4L(2N+1)^d\exp\Big(-\frac{Nu^2}{18(L+1)^2\beta^2_{\max}}\Big)+4(2N+1)^d\exp\Big(-\frac{N\big(\frac{u}{3}\big)^2}{2(L+1)^2\beta^2_{L+1}}\Big) \\
&= 4(L+1)(2N+1)^d\exp\Big(-\frac{Nu^2}{18(L+1)^2(\beta_{\max}\vee\beta_{L+1})^2}\Big),
}
which verifies the inductive step needed for \eqref{first_summands}.
\end{proof}

\section{Proof of Theorem \ref{weaker_main_thm}} \label{weaker_proof}

In preparation for the proof, we introduce the main input, Theorem \ref{previous_thm}, from \cite{bates-chatterjee20II}.
The statement is exactly the same as Theorem \ref{previous_thm_polymer} but holds for more general Gaussian disordered systems.
So that there is no confusion caused by duplicate notation, let us introduce a generic setting.

Let $(\Omega,\FF,\Pb)$ be an abstract probability space, and $(\Sigma_N)_{N\geq1}$ a sequence of Polish spaces %
equipped respectively with probability measures $(\nu_N)_{N\geq1}$. 
For each $N$, we consider a centered Gaussian field $(\HH_N(\sigma))_{\sigma\in\Sigma_N}$ defined on $\Omega$.  
Regarding this field as a Hamiltonian, we denote the associated Gibbs measure by
\eeq{ \label{nu_def}
\nu_{N,\beta}(\dd\sigma) \coloneqq \frac{\e^{\beta \HH_N(\sigma)}}{\ZZ_N(\beta)}\ \nu_N(\dd\sigma), \quad \text{where} \quad
\ZZ_N(\beta) \coloneqq \int_{\Sigma_N} \e^{\beta \HH_N(\sigma)}\ \nu_N(\dd\sigma).
}
We make the following assumptions:
\begin{itemize}
\item There is a deterministic function $\PPP \colon [0,\infty) \to \R$ and a deterministic sequence $(a_N)_{N\geq1}$ tending to infinity,\footnote{Strictly speaking, \cite{bates-chatterjee20II} considers only the case $a_N=N$, although this is just for purposes of exposition. Even so, this single case would be enough for our purposes, since we will ultimately apply Theorem \ref{previous_thm} with $a_N = n_{\ell}(N)-n_{\ell-1}(N)$.  Indeed, the associated sequence of partitions $(\PP_{N,L})_{N\geq1}$ from \eqref{partition_def} is contained in the union of finitely many sequences of the form $(\PP'_{N_i,L})_{i\geq1}$, where $n'_\ell(N_i)-n'_{\ell-1}(N_i) = i$.  Therefore, one can safely apply Theorem \ref{previous_thm} along each one of these sequences, and then the proof of Corollary \ref{ptwise_thm} goes through by choosing the maximum $J$, maximum $N_*$, and minimum $\delta$ resulting from these applications.}
such that
\[\label{free_energy_assumption} \tag{A1}
\lim_{N\to\infty} \frac{\log \ZZ_N(\beta)}{a_N} = \PPP(\beta) \quad \Pb\text{-}\mathrm{a.s.} \text{ and in $L^1(\Pb)$, for every $\beta\geq0$}. 
\]
\item For every $\sigma\in\Sigma_N$, we have
\[ \label{variance_assumption} \tag{A2}
\Var \HH_N(\sigma) = a_N.
\]
\item For any $\sigma^1,\sigma^2\in\Sigma_N$, we have
\[ \label{positive_overlap} \tag{A3}
\RR_N(\sigma^1,\sigma^2) \coloneqq \Corr(\HH_N(\sigma^1),\HH_N(\sigma^2)) = \frac{1}{a_N}\Cov(\HH_N(\sigma^1),\HH_N(\sigma^2)) \geq 0.
\]
\item For each $N$, there exist measurable real-valued functions $(\vphi_{i,N})_{i=1}^\infty$ on $\Sigma_N$ 
and i.i.d.~standard normal random variables $(\mathfrak{g}_{i,N})_{i=1}^\infty$ defined on $\Omega$ such that for each $\sigma\in\Sigma_N$,
\[ \label{field_decomposition} \tag{A4}
\HH_N(\sigma)  = \sum_{i=1}^\infty \mathfrak{g}_{i,N}\vphi_{i,N}(\sigma) \qquad \text{$\Pb$-$\mathrm{a.s.}$},
\]
where the series on the right converges in $L^2(\Pb)$. 
\end{itemize}

\begin{theirthm} \label{previous_thm}
\textup{\cite[Thm.~1.2]{bates-chatterjee20II}}
Assume \eqref{free_energy_assumption}--\eqref{field_decomposition}.
If $\beta>0$ is a point of differentiability for $\PPP(\cdot)$ with $\PPP'(\beta) < \beta$,
then for every $\eps > 0$, there exist integers $J = J(\beta,\eps)$ and $N_* = N_*(\beta,\eps)$ and a number $\delta = \delta(\beta,\eps)>0$ such that the following is true for all $N\geq N_*$.
With $\Pb$-probability at least $1-\eps$, there exist $\sigma^1,\dots,\sigma^J\in\Sigma_N$ such that 
\eq{
\nu_{N,\beta}\bigg(\bigcup_{j=1}^J \big\{\sigma\in\Sigma_N:\, R(\sigma^{j},\sigma)\geq\delta\big\}\bigg) \geq 1 - \eps.
}
\end{theirthm}

Returning to the polymer setting, we consider the following modifications to the random environment: restricting to times inside the interval $\llbrack n_{\ell-1}+1,n_\ell\rrbrack$, and restricting to times outside the interval.
The resulting Hamiltonians will be written as
\eq{
H_N^{(\ell)}(\sigma) &\coloneqq \sum_{i \in \llbrack n_{\ell-1}+1,n_\ell\rrbrack} g_N(i,\sigma_i), \\
\wh H_N^{(\ell)}(\sigma) &\coloneqq \sum_{i \in \llbrack 1,N\rrbrack \setminus \llbrack n_{\ell-1}+1,n_\ell\rrbrack} g_N(i,\sigma_i),
}
where $(g_N(i,x):\, N\geq1, i\in\llbrack1,N\rrbrack, x\in\Z^d)$ is an i.i.d.~collection even across $N$ (recall Remark \ref{independence_remark}).
Therefore, we can partition $\vc g$ into the following pair of independent sub-collections,
\eq{
\vc g^{(\ell)} &\coloneqq \{g_N(i,x) :\, N\geq 1, i\in\llbrack n_{\ell-1}+1,n_\ell \rrbrack,x\in\Z^d\}, \\
\wh{\vc g}^{(\ell)} &\coloneqq \{g_N(i,x) :\, N\geq1, i\in\llbrack 1,N\rrbrack\setminus\llbrack n_{\ell-1}+1,n_\ell\rrbrack,x\in\Z^d\},
}
and then $H_N^{(\ell)}(\cdot)$ is a function of $\vc g^{(\ell)}$, while $\wh H_N^{(\ell)}(\cdot)$ is a function of $\wh{\vc g}^{(\ell)}$.
We will write $\P_{\wh{\vc g}^{(\ell)}}$ to denote the probability measure obtained by conditioning $\P$ on $\wh{\vc g}^{(\ell)}$, and $\E_{\wh{\vc g}^{(\ell)}}$ will denote expectation with respect to $\P_{\wh{\vc g}^{(\ell)}}$ (\textit{i.e.}~integrating over just $\vc g^{(\ell)}$).
While the law of $\vc g^{(\ell)}$ is no different under $\P_{\wh{\vc g}^{(\ell)}}$ than under $\P$, these notational devices will make clearer how we invoke Theorem \ref{previous_thm} and avoid the slightly more cumbersome $\P\givenp{\,\cdot}{\wh{\vc g}^{(\ell)}}$.

Next, for each $\ell\in\llbrack 1,L\rrbrack$ we introduce the following probability measure on $\Sigma$:
\eq{
\wh \mu_{N,\beta}^{(\ell)}(\dd\sigma) = \frac{1}{\wh Z_N^{(\ell)}(\beta)}\e^{\beta\wh H_N^{(\ell)}(\sigma)}\ P(\dd\sigma), \qquad
\wh Z_N^{(\ell)}(\beta) \coloneqq E(\e^{\beta\wh H_N^{(\ell)}(\sigma)}).
}
Observe that
\eeq{ \label{mu_def_2}
\mu_{N,\beta}(\dd\sigma) &= \frac{1}{Z_N(\beta)}\e^{\beta H_N^{(\ell)}(\sigma)}\e^{\beta \wh H_N^{(\ell)}(\sigma)}\ P(\dd\sigma) \\
&= \frac{\wh Z_N^{(\ell)}(\beta)}{Z_N(\beta)}\e^{\beta H_N^{(\ell)}(\sigma)}\ \wh \mu_{N,\beta}^{(\ell)}(\dd\sigma)
= \frac{1}{Z_N^{(\ell)}(\beta)}\e^{\beta H_N^{(\ell)}(\sigma)}\ \wh \mu_{N,\beta}^{(\ell)}(\dd\sigma),
}
where
\eeq{ \label{without_hat}
Z_N^{(\ell)}(\beta) &\coloneqq 
\frac{Z_N(\beta)}{\wh Z_N^{(\ell)}(\beta)}.
}
We will ultimately apply Theorem \ref{previous_thm} to this ``restricted'' setting, using
\begin{linenomath}\postdisplaypenalty=0
\begin{align} \label{identifications}
\Sigma_N&=\Sigma \text{ from \eqref{sigma_def}}, &
\nu_N &= \wh\mu_{N,\beta}^{(\ell)}, &
\Pb&=\P_{\wh{\vc g}^{(\ell)}}, &
\HH_N&=H_N^{(\ell)}, \\
a_N&=n_{\ell}-n_{\ell-1}, &
\nu_{N,\beta}&=\mu_{N,\beta}, &
\ZZ_{N}(\beta)&=Z_N^{(\ell)}(\beta), &
\RR_{N}(\cdot,\cdot)&=R^{(\ell)}(\cdot,\cdot). \notag
\end{align}
\end{linenomath}
Note that $\nu_N$ and $\Pb$ are now random measures depending on $\wh{\vc g}^{(\ell)}$, but since the Hamiltonian $\HH_N$ is independent of this randomness, Theorem \ref{previous_thm} will still apply.
With these choices, we first need to verify \eqref{free_energy_assumption}.

\begin{prop} \label{conditional_free_energy_thm}
For each $\ell \in\llbrack 1,L\rrbrack$ and any $\beta\geq0$, the following statement holds $\P$-almost surely:
\eeq{ \label{without_hat_to_show}
\lim_{N\to\infty} \frac{\log Z_N^{(\ell)}(\beta)}{n_\ell-n_{\ell-1}} = p(\beta) \quad \text{$\P_{\wh{\vc g}^{(\ell)}}$-$\mathrm{a.s.}$ and in $L^\alpha(\P_{\wh{\vc g}^{(\ell)}})$ for all $\alpha\in[1,\infty)$}.
}
\end{prop}

\begin{proof}
If $\beta = 0$, then $Z_N^{(\ell)}(0) = 1$ is deterministic, and so \eqref{without_hat_to_show} holds trivially with $p(0)=0$.
Consequently, we may assume $\beta>0$.

By Theorem \ref{generalized_free_energy_thm}, we know
\eeq{ \label{full_and_hat}
\lim_{N\to\infty} \frac{\log Z_N(\beta)}{N} = p(\beta), \qquad \lim_{N\to\infty} \frac{\log \wh Z_N^{(\ell)}(\beta)}{N} = \frac{L-1}{L} p(\beta),
}
where the limits are $\P$-almost sure and in $L^\alpha(\P)$ for every $\alpha\in[1,\infty)$.
By definition \eqref{without_hat} and the fact that $(n_{\ell}-n_{\ell-1}) \sim N/L$, the following limit thus holds in the same senses:
\eeq{ \label{partial_limit}
\lim_{N\to\infty} \frac{\log Z_N^{(\ell)}(\beta)}{n_\ell-n_{\ell-1}} = p(\beta).
}
In particular, since \eqref{partial_limit} holds $\P$-almost surely, Fubini's theorem guarantees the following: For almost every realization of $\wh{\vc g}^{(\ell)}$, \eqref{partial_limit} holds $\P_{\wh{\vc g}^{(\ell)}}$-almost surely.
This proves the first part of \eqref{without_hat_to_show}.

Meanwhile, for any $\eps>0$ and $\alpha\geq 1$, Markov's inequality gives
\eq{
&\P\Big(\E_{\wh{\vc g}^{(\ell)}} \Big|\frac{\log Z_N(\beta)}{N} -\frac{\E\log Z_N(\beta)}{N}\Big|^\alpha>\eps\Big)
\leq \eps^{-1}\E\Big|\frac{\log Z_N(\beta)}{N} - \frac{\E\log Z_N(\beta)}{N}\Big|^\alpha \\
&\stackrefp{free_energy_thm_2}{=} \eps^{-1}\int_0^\infty \P\Big(\Big|\frac{\log Z_N(\beta)}{N} - \frac{\E\log Z_N(\beta)}{N}\Big|^\alpha>u\Big)\ \dd u \\
&\stackref{free_energy_thm_2}{\leq} 2\eps^{-1}\int_0^\infty \exp\Big(-\frac{Nu^{2/\alpha}}{2\beta^2}\Big)\ \dd u
= CN^{-\alpha/2},
}
where $C$ depends on $\alpha$ and $\beta$ but not on $N$.
By taking $\alpha > 2$, we can apply Borel--Cantelli to determine that with $\P$-probability one,
\eq{
\limsup_{N\to\infty} \E_{\wh{\vc g}^{(\ell)}} \Big|\frac{\log Z_N(\beta)}{N} - \frac{\E\log Z_N(\beta)}{N}\Big|^\alpha \leq \eps.
}
By taking a countable sequence $\eps_k\searrow0$, we further deduce
\eeq{ \label{further_deduction}
\lim_{N\to\infty} \E_{\wh{\vc g}^{(\ell)}} \Big|\frac{\log Z_N(\beta)}{N} - \frac{\E\log Z_N(\beta)}{N}\Big|^\alpha = 0 \qquad \text{$\P$-$\mathrm{a.s.}$}
}
Since \eqref{free_energy_thm_1} gives the deterministic limit $\frac{1}{N}\E\log Z_N(\beta)\to p(\beta)$, it follows from \eqref{further_deduction} that with $\P$-probability one we have 
\begin{subequations} \label{Lalpha}
\eeq{ 
\lim_{N\to\infty}\frac{1}{N}\log Z_N(\beta)= p(\beta) \quad \text{in $L^\alpha(\P_{\wh{\vc g}^{(\ell)}})$}.
}
Moreover, given that $\alpha$ can be taken arbitrarily large, this convergence occurs in $L^\alpha(\P_{\wh{\vc g}^{(\ell)}})$ simultaneously for all $\alpha\in[1,\infty)$.
On the other hand, from \eqref{full_and_hat} we know
\eeq{ 
\lim_{N\to\infty}\frac{1}{N}\log \wh Z_N^{(\ell)}(\beta) = \frac{L-1}{L}p(\beta),
}
\end{subequations}
also with $\P$-probability one.
Furthermore, since $\wh Z_n^{(\ell)}(\beta)$ is determined entirely by $\wh{\vc g}^{(\ell)}$, this last limit is a deterministic statement with respect to $\P_{\wh{\vc g}^{(\ell)}}$; in particular, it holds in $L^\alpha(\P_{\wh{\vc g}^{(\ell)}})$.
From \eqref{without_hat}, \eqref{Lalpha}, and the fact that $(n_{\ell}-n_{\ell-1}) \sim N/L$, we now have
\eq{
\lim_{N\to\infty} \E_{\wh{\vc g}^{(\ell)}}\Big|\frac{\log Z_N^{(\ell)}(\beta)}{n_\ell-n_{\ell-1}} - p(\beta)\Big|^\alpha = 0 \qquad \text{$\P$-$\mathrm{a.s.}$}
}
We have thus verified both parts of \eqref{without_hat_to_show}.
\end{proof}

Given Proposition \ref{conditional_free_energy_thm}, we can make a statement approaching Theorem \ref{weaker_main_thm}.
The following result asserts that once the system size becomes large enough, the ``external'' disorder $\wh{\vc g}^{(\ell)}$ becomes sufficiently well behaved so that when only the ``internal'' disorder $\vc g^{(\ell)}$ is regarded as random, the polymer along the subinterval $\llbrack n_{\ell-1}+1,n_\ell\rrbrack$ admits the same localization statement as in Theorem \ref{previous_thm_polymer}.

\begin{cor} \label{ptwise_thm}
Let $\ell\in\llbrack 1,L\rrbrack$, and assume $\beta>0$ is a point of differentiability for $p(\cdot)$ with $p'(\beta) < \beta$.
Then with $\P$-probability one, the following is true.
For every $\eps > 0$, there exist integers $J = J(\beta,\eps,\wh{\vc g}^{(\ell)})$ and $N_* = N_*(\beta,\eps,\wh{\vc g}^{(\ell)})$ and a number $\delta = \delta(\beta,\eps,\wh{\vc g}^{(\ell)})>0$ such that for all $N\geq N_*$, the following event has
$\P_{\wh{\vc g}^{(\ell)}}$-probability at least $1-\eps$:
\eq{
\Esf_{N,J,\delta,\eps}^{(\ell)} \coloneqq \bigg\{\exists\, \sigma^1,\dots,\sigma^J\in\Sigma :\, \mu_{N,\beta}\bigg(\bigcup_{j=1}^J \big\{\sigma\in\Sigma:\, R^{(\ell)}(\sigma^{j},\sigma)\geq\delta\big\}\bigg) \geq 1 - \eps\bigg\}.
}
\end{cor}

\begin{proof}
Because $\wh{\vc g}^{(\ell)}$ and $\vc g^{(\ell)}$ are independent, the law of the latter given the former remains i.i.d.~standard normal.
Therefore, \eqref{mu_def_2} is a representation of $\mu_{N,\beta}$ in the form of \eqref{nu_def}.
Proposition \ref{conditional_free_energy_thm} verifies that in this representation, the assumption \eqref{free_energy_assumption} holds.
Also, it is trivial to check that
\eq{
\Cov(H_N^{(\ell)}(\sigma^1),H_N^{(\ell)}(\sigma^2)) = (n_\ell-n_{\ell-1})R^{(\ell)}(\sigma^1,\sigma^2) \geq 0, \quad \sigma^1,\sigma^2\in\Sigma.
}
In particular, we have
\eeq{ \label{variance_assumption_new}
\Var H_N^{(\ell)}(\sigma) = n_{\ell}-n_{\ell-1}, \quad \sigma\in\Sigma.
}
Thus \eqref{variance_assumption}--\eqref{field_decomposition} also hold, and we can
apply Theorem \ref{previous_thm} with the identifications in \eqref{identifications} to obtain the result.
%
\end{proof}

Recall that $\vc g$ is supported on a probability space we denote $(\Omega,\FF,\P)$.
Following Corollary \ref{ptwise_thm}, we can define the event
\eeq{ \label{hat_event}
\wh \Esf_{N,J,\delta,\eps}^{(\ell)} \coloneqq \{\P_{\wh{\vc g}^{(\ell)}}(\Esf_{N,J,\delta,\eps}^{(\ell)} ) \geq 1-\eps\}.
}
Let us postpone verification that such an event is measurable, and proceed directly to the proof of Theorem \ref{weaker_main_thm}.

\begin{proof}[Proof of Theorem \ref{weaker_main_thm}]
Let $\eps > 0$ and $L$ be given.
In the notation of Corollary \ref{ptwise_thm}, it suffices to find $J$, $N_*$, and $\delta$ such that
\eeq{ \label{weaker_main_thm_to_show}
\P\bigg(\bigcap_{\ell=1}^L \Esf_{N,J,\delta,\eps/L}^{(\ell)}\bigg) \geq 1-\eps \quad \text{for all $N\geq N_*$,}
}
since the event $\bigcap_{\ell=1}^L \Esf_{N,J,\delta,\eps/L}^{(\ell)}$ implies the existence of $\sigma^1,\dots,\sigma^{JL} \in \Sigma$ such that
\eq{
\mu_{N,\beta}\bigg(\bigcap_{\ell=1}^L\bigcup_{j=1}^{JL} \{\sigma\in\Sigma :\, R^{(\ell)}(\sigma^j,\sigma) \geq \delta\}\bigg)\geq 1-\eps.
}
The conclusion of Theorem \ref{weaker_main_thm} then follows by replacing $J$ with $JL$.
The remainder of the proof is thus establishing \eqref{weaker_main_thm_to_show}.

Let $\eps' = \eps'(\eps,L)$ be a positive number to be specified later. 
Take $(\delta_k)_{k\geq1}$ to be any decreasing sequence tending to $0$ as $k\to\infty$.
From Corollary \ref{ptwise_thm}, we know
\eq{
\P\bigg(\bigcup_{J = 1}^\infty\bigcup_{k=1}^\infty \bigcup_{N_*=1}^\infty\bigcap_{N=N_*}^\infty \wh \Esf_{N,J,\delta_k,\eps'}^{(\ell)}\bigg) = 1 \quad \text{for each $\ell\in\llbrack 1,L\rrbrack$}.
}
Since $\wh \Esf_{N,J,\delta_k,\eps'} \subset \wh \Esf_{N,J+1,\delta_k,\eps'}$, we can choose $J = J(\beta,\eps',L)$ sufficiently large that
\eq{
\P\bigg(\bigcup_{k=1}^\infty \bigcup_{N_*=1}^\infty\bigcap_{N=N_*}^\infty \wh \Esf_{N,J,\delta_k,\eps'}^{(\ell)}\bigg) \geq 1 - \eps' \quad \text{for each $\ell\in\llbrack 1,L\rrbrack$}.
}
By the assumption $\delta_{k}>\delta_{k+1}$, we also have $\wh \Esf_{N,J,\delta_k,\eps'}^{(\ell)} \subset \wh \Esf_{N,J,\delta_{k+1},\eps'}^{(\ell)}$, and so we can choose $k = k(\beta,\eps',L,J)$ sufficiently large that
\eq{
\P\bigg(\bigcup_{N_*=1}^\infty\bigcap_{N=N_*}^\infty \wh \Esf_{N,J,\delta_k,\eps'}^{(\ell)}\bigg) \geq 1 - 2\eps' \quad \text{for each $\ell\in\llbrack 1,L\rrbrack$}.
}
Henceforth we simply write $\delta = \delta_k$.
Finally, we choose $N_* = N_*(\beta,\eps',L,J,\delta)$ sufficiently large that
\eq{
\P\bigg(\bigcap_{N=N_*}^\infty \wh \Esf_{N,J,\delta,\eps'}^{(\ell)}\bigg) \geq 1 - 3\eps' \quad \text{for each $\ell\in\llbrack 1,L\rrbrack$}.
}
Now, whether or not $\wh \Esf_{N,J,\delta,\eps'}^{(\ell)}$ occurs depends only on $\wh{\vc g}^{(\ell)}$; see Proposition \ref{measurability_prop}.
By definition \eqref{hat_event}, when $\wh \Esf_{N,J,\delta,\eps'}^{(\ell)}$ does occur, the event $\Esf_{N,J,\delta,\eps'}^{(\ell)}$ has $\P_{\wh{\vc g}^{(\ell)}}$-probability at least $1-\eps'$.
Therefore, by our choice of $N_*$, we have
\eq{
\P(\Esf_{N,J,\delta,\eps'}^{(\ell)}) &= \E(\P_{\wh{\vc g}^{(\ell)}}(\Esf_{N,J,\delta,\eps'}^{(\ell)})) \\
&\geq \E\big(\one_{\wh \Esf_{N,J,\delta,\eps'}^{(\ell)}}\P_{\wh{\vc g}^{(\ell)}}(\Esf_{N,J,\delta,\eps'}^{(\ell)})\big) \\
&\geq (1-3\eps')(1-\eps') \geq 1-4\eps' \quad \text{for each $\ell\in\llbrack 1,L\rrbrack$, $N\geq N_*$},
}
and thus
\eq{
\P\bigg(\bigcap_{\ell=1}^L \Esf_{N,J,\delta,\eps'}^{(\ell)}\bigg) \geq 1-4L\eps' \quad \text{for all $N\geq N_*$}.
}
To complete the proof, we set $\eps' = \eps/(4L)$ and note that
\eq{
\P\bigg(\bigcap_{\ell=1}^L \Esf_{N,J,\delta,\eps}^{(\ell)}\bigg)
\geq \P\bigg(\bigcap_{\ell=1}^L \Esf_{N,J,\delta,\eps'}^{(\ell)}\bigg) \geq 1-\eps \quad \text{for all $N\geq N_*$}.
\tag*{\qedhere}
}
\end{proof}

To conclude the section, we return to the technical issue of measurability for the event $\wh \Esf^{(\ell)}_{N,J,\delta,\eps}$ defined in \eqref{hat_event}.
Let $\wh\FF^{(\ell)}$ denote the sub-sigma-algebra of $\FF$ generated by $\wh{\vc g}^{(\ell)}$.

\begin{prop} \label{measurability_prop}
For any integers $N \geq L$, $J\geq 1$, $\ell\in\llbrack 1,L\rrbrack$, and numbers $\delta,\eps>0$, we have $\wh \Esf_{N,J,\delta,\eps}^{(\ell)} \in \wh\FF^{(\ell)}$.
\end{prop}

We will make use of two lemmas.

\begin{lemma} \label{rational_lemma}
If $r$ is a non-constant rational function in $n$ real variables, then for any $t\in\R$, the set $\{\vc x\in\R^n :\, r(\vc x) = t\}$ has zero Lebesgue measure.
\end{lemma}

\begin{proof}
Let us write $r = f/g$, where $f$ and $g$ are polynomials.
Then $r - t = (f-tg)/g$, which vanishes if and only if $f - tg = 0$.
By hypothesis, $f - tg$ is not identically equal to $0$, and so this polynomial may only vanish on a set of Lebesgue measure zero \cite{eldredge16}.
\end{proof}

\begin{lemma} \label{continuous_lemma}
Let $\vc X \in \R^n$ be a random vector supported on $(\Omega,\FF,\P)$.
Suppose that $f \colon \R^{n+m} \to \R$ is a continuous function such that
\eq{
\P(f(\vc X,\vc y) = t) = 0 \quad \text{for any $\vc y\in\R^m$, $t\in\R$}.
}
Then the map
$\vc y \mapsto \P(f(\vc X,\vc y) \geq t)$
is continuous from $\R^m$ to $[0,1]$, for any $t\in\R$.
\end{lemma}

\begin{proof}
Fix $\vc y\in\R^m$ and $t\in\R$, and let $\eps>0$ be given.
By hypothesis, we can choose $h>0$ so small that $\P(f(\vc X,\vc y) \in [t - h,t+h]) < \eps$.
Next choose $K\subset\R^n$ to be a compact set sufficiently large that $\P(\vc X\notin K) < \eps$.
By uniform continuity of $f$ on $K \times (\vc y + [-1,1]^m)$, we may choose $\delta > 0$ sufficiently small that
\eq{
\vc X \in K, \vc\delta\in\R^m, \|\vc\delta\| < \delta \quad \implies \quad |f(\vc X,\vc y+\vc \delta) - f(\vc X,\vc y)| \leq h.
}
It now follows that whenever $\|\vc\delta\| < \delta$, we have
\eq{
\P(f(\vc X,\vc y+\vc \delta) \geq t) &\geq \P(\vc X \in K, f(\vc X,\vc y) \geq t + h) 
\geq \P(f(\vc X,\vc y) \geq t) - 2\eps,
}
as well as
\eq{
\P(f(\vc X,\vc y+\vc \delta) < t)  
\geq \P(\vc X\in K,f(\vc X,\vc y) < t-h) 
\geq \P(f(\vc X,\vc y) < t) - 2\eps.
}
The two previous displays together imply
\eq{
|\P(f(\vc X,\vc y+\vc \delta) \geq t) - \P(f(\vc X,\vc y) \geq t)| \leq 2\eps.
}
As $\eps$ is arbitrary, the continuity of the map $\vc y \mapsto \P(f(\vc X,\vc y) \geq t)$ has been proved.
\end{proof}

\begin{proof}[Proof of Proposition \ref{measurability_prop}]

Recall the set $D_i = \{x\in\Z^d :\, P(\sigma_i = x) > 0\}$, for $i\geq1$.
For every finite $N$, the random Hamiltonian $H_N(\cdot)$ depends only on a finite set of random variables, namely
\eq{
\vc g_N \coloneqq \{g_N(i,x) :\, 1\leq i\leq N, x\in D_i\}.
}
As before, let us partition this collection is the following disjoint sub-collections:
\eq{
\vc g_N^{(\ell)} &\coloneqq \{g_N(i,x) :\, i\in\llbrack n_{\ell-1}+1, n_\ell\rrbrack, x\in D_i\}, \\
\wh{\vc g}_N^{(\ell)} &\coloneqq \{g_N(i,x) :\, i\in\llbrack 1,N\rrbrack\setminus\llbrack n_{\ell-1}+1, n_\ell\rrbrack, x\in D_i\}. 
}
We will restrict our attention to these finite collections and thus regard all subsequent statements as concerning only finite-dimensional vectors.
In keeping with this finite-dimensional perspective, we will write $\Sigma_N$ to denote the set of the $(2d)^N$ possible simple random walk paths starting at the origin and consisting of $N$ steps.
It is natural to regard $\mu_{N,\beta}$ and $\wh\mu_{N,\beta}^{(\ell)}$ as measures on $\Sigma_N$, as opposed to $\Sigma$ defined in \eqref{sigma_def}.
Similarly, it makes sense to take \eqref{ell_overlap_def} as a definition of $R^{(\ell)}(\sigma^1,\sigma^2)$ for $\sigma^1,\sigma^2\in\Sigma_N$.

Now consider subsets of $\Sigma_N$ of the form
\eq{
S^{(\ell)}_{\sigma^1,\dots,\sigma^J,\delta} \coloneqq \bigcup_{j=1}^J \{\sigma\in\Sigma_N :\, R^{(\ell)}(\sigma^j,\sigma) \geq \delta\}, \quad \sigma^1,\dots,\sigma^J \in \Sigma_N,  \delta > 0, \ell\in\llbrack 1,L\rrbrack.
}
The event $\Esf^{(\ell)}_{N,J,\delta,\eps}$ from Corollary \ref{ptwise_thm} can be expressed as
\eq{
\Esf^{(\ell)}_{N,J,\delta,\eps} &= \bigcup_{\sigma^1,\dots,\sigma^J\in\Sigma_N} \{\mu_{N,\beta}(S_{\sigma^1,\dots,\sigma^J,\delta})\geq1-\eps\} \\
&= \Big\{\max_{\sigma^1,\dots,\sigma^J\in\Sigma_N} \mu_{N,\beta}(S_{\sigma^1,\dots,\sigma^J,\delta})\geq1-\eps\Big\} \\
&= \bigcap_{n = 1}^\infty \bigcup_{q=1}^\infty \bigg\{\bigg(\frac{1}{|\Sigma_N|^J}\sum_{\sigma^1,\dots,\sigma^J\in\Sigma_N} \mu_{N,\beta}(S_{\sigma^1,\dots,\sigma^J,\delta})^q\bigg)^{1/q} \geq 1 - \eps - \frac{1}{n}\bigg\}.
}
By monotonicity, it is evident that
\eeq{ \label{limit_of_continuous}
&\P_{\wh{\vc g}^{(\ell)}}(\Esf^{(\ell)}_{N,J,\delta,\eps}) \\
&= \lim_{n\to\infty}\lim_{q\to\infty} \P_{\wh{\vc g}^{(\ell)}}\bigg(\bigg(\frac{1}{|\Sigma_N|^J}\sum_{\sigma^1,\dots,\sigma^J\in\Sigma_N} \mu_{N,\beta}(S_{\sigma^1,\dots,\sigma^J,\delta})^q\bigg)^{1/q} \geq 1 - \eps - \frac{1}{n}\bigg).
}
Now we observe that the quantity
\eq{
\frac{1}{|\Sigma_N|^J}\sum_{\sigma^1,\dots,\sigma^J\in\Sigma_N} \mu_{N,\beta}(S_{\sigma^1,\dots,\sigma^J,\delta})^q
}
is a non-constant rational function in the variables $\{\e^{g(i,x)} :\, i\geq1,x\in D_i\}$.
Moreover, it remains non-constant even if the realization of $\wh{\vc g}^{(\ell)}_N$ is fixed (assuming $n_{\ell-1}(N) < n_{\ell}(N)$, which is true so long as $N\geq L$).
Therefore, Lemma \ref{rational_lemma} tells us that the map
\eq{
({\vc g}^{(\ell)}_N,\wh{\vc g}^{(\ell)}_N) \mapsto \bigg(\frac{1}{|\Sigma_N|^J}\sum_{\sigma^1,\dots,\sigma^J\in\Sigma_N} \mu_{N,\beta}(S_{\sigma^1,\dots,\sigma^J,\delta})^q\bigg)^{1/q}
}
satisfies the hypotheses of Lemma \ref{continuous_lemma} (note that $\e^{g(i,x)}$ is a continuous function of $g(i,x)$ and has a density with respect to Lebesgue measure).
In turn, Lemma \ref{continuous_lemma} guarantees the continuity---in particular, measurability---of the map
\eq{
\wh{\vc g}^{(\ell)} \mapsto \P_{\wh{\vc g}^{(\ell)}} \bigg(\bigg(\frac{1}{|\Sigma_N|^J}\sum_{\sigma^1,\dots,\sigma^J\in\Sigma_N} \mu_{N,\beta}(S_{\sigma^1,\dots,\sigma^J,\delta})^q\bigg)^{1/q} \geq 1 - \eps- \frac{1}{n}\bigg).
}
Consequently, \eqref{limit_of_continuous} exhibits $\wh{\vc g}^{(\ell)} \mapsto \P_{\wh{\vc g}^{(\ell)}}(\Esf^{(\ell)}_{N,J,\delta,\eps})$ as a limit of measurable functions, implying that this map is also measurable.
In particular, the event $\wh \Esf_{N,J,\delta,\eps}^{(\ell)}$ considered in \eqref{hat_event} is $\wh\FF^{(\ell)}$-measurable.
\end{proof}

\section{Acknowledgments}
I am indebted to Sourav Chatterjee for suggesting the consideration of a time-dependent inverse temperature.
This idea was the inspiration leading to the present article.
I also thank Francis Comets for valuable discussion during the workshop on ``Self-interacting random walks, Polymers and Folding'' held at Centre International de Rencontres Math\'ematiques, and Hubert Lacoin for directing my attention to \cite[Sec.~7]{carmona-hu02}.
Finally, I am very grateful to the referees for their corrections and suggestions, which improved the exposition.

\bibliography{path_localization}

\end{document}